\documentclass{article}

\usepackage{arxiv}

\usepackage[utf8]{inputenc} 
\usepackage[T1]{fontenc}    
\usepackage[colorlinks,citecolor=blue,urlcolor=blue]{hyperref} 
\usepackage{url}            
\usepackage{booktabs}       
\usepackage{amsfonts}       
\usepackage{nicefrac}       
\usepackage{microtype}      
\usepackage{lipsum}		
\usepackage{doi}
\usepackage[ruled,vlined]{algorithm2e}
\usepackage{float}
\usepackage{graphicx}
\usepackage{natbib}
\usepackage{amsthm,amsmath,amsfonts,amssymb,amsbsy}
\usepackage{caption}
\usepackage{subcaption}
\usepackage{xcolor}
\usepackage{tikz}
\usetikzlibrary{positioning, fit, calc, shapes, arrows, matrix}
\usepackage{duckuments}
\usepackage{tabularray}
\UseTblrLibrary{booktabs}
\usepackage{graphicx}
\usepackage{pdfpages}

\usepackage{etoolbox} 
\usepackage{float} 
\usepackage{tikz}
\usetikzlibrary{shapes.geometric, positioning, matrix, fit, arrows.meta}
\usepackage[outline]{contour}
\contourlength{1.2pt}

\expandafter\let\csname proof\endcsname\relax
\expandafter\let\csname endproof\endcsname\relax

\newcommand{\xX}{\mathbf{x}}
\newcommand{\xI}{\mathcal{T}}
\newcommand{\xR}{\mathbb{R}}
\newcommand{\dimI}{d_\xI}
\newcommand{\dimD}{d_D}
\newcommand{\dxx}{d_{\xX, \xX'}}

\newcommand{\kincrement}{k_{\text{inc}}}
\newcommand{\mincrement}{m_{\text{inc}}}
\newcommand{\xM}{M(\xX, \xX')}
\newcommand{\diam}{D_{\xX, \xX'}(\xI)}
\newcommand{\xY}{\Vert \xX - \xX' \Vert^{\alpha_1 /2}}
\newcommand{\sigFieldM}{\mathcal{B}(\xI)}

\newcommand{\indpospdf}{\mathcal{A}^+}
\newcommand{\indpdf}{\mathcal{A}}
\newcommand{\pospdf}{\mathcal{A}^+}
\newcommand{\pdf}{\mathcal{A}}

\newcommand{\logt}{\psi} 
\newcommand{\slogt}{\Psi} 

\newcommand{\proc}[3]{
    \ifstrempty{#3}%
    {%
        \ifstrempty{#2}%
        {%
            #1
        }{%
            #1_{#2}
        }%
    }{%
        (#1_{#2})_{#3}
    }%
}
\newcommand{\SLGPM}[2]{
    \ifstrempty{#1}%
        {%
            \Xi%
        }{%
            \ifstrempty{#2}%
            {%
                \Xi_{#1}%
            }{%
                \Xi_{#1}\left( #2 \right)%
            }%
        }%
} 

\newtheoremstyle{TheoremNum}
    {\topsep}{\topsep}              
    {\itshape}                      
    {}                              
    {\bfseries}                     
    {.}                             
    { }                             
    {\thmname{#1}\thmnote{ \bfseries #3}}
\theoremstyle{TheoremNum}
    \newtheorem{thmn}{Theorem} 
\newtheoremstyle{PropositionNum}
    {\topsep}{\topsep}              
    {\itshape}                      
    {}                              
    {\bfseries}                     
    {.}                             
    { }                             
    {\thmname{#1}\thmnote{ \bfseries #3}}
\theoremstyle{PropositionNum}
    \newtheorem{propn}{Proposition}
    \newtheorem{condn}{Condition}

\usepackage{xr}
\theoremstyle{plain}
\newtheorem{theorem}{Theorem}[section]
\newtheorem{lemma}{Lemma}
\newtheorem{proposition}[theorem]{Proposition}
\newtheorem{corollary}[theorem]{Corollary}

\theoremstyle{definition}
\newtheorem{definition}{Definition}[section]
\newtheorem{condition}{Condition}

\theoremstyle{remark}
\newtheorem*{remark}{Remark}
\newtheorem*{proof}{Proof}

\title{Continuous logistic Gaussian random measure fields for spatial distributional modelling}

\author{Athénaïs Gautier$^1$
\and David Ginsbourger$^2$\\$^1$ DTIS, ONERA, Université Paris-Saclay, 91120, Palaiseau, France \and $^2$ Institute of Mathematical Statistics and Actuarial Science, University of Bern, Switzerland}

\renewcommand{\shorttitle}{\textit{arXiv} Template}

\hypersetup{
pdftitle={SLGP-continuous},
pdfauthor={Athénaïs Gautier},
pdfkeywords={Logistic Gaussian Process \and Spatial statistics \and Non-parametric models \and Spatial regularity \and Random measures  \and Distributional learning \and Covariance kernels},
}

\begin{document}
\maketitle

\begin{abstract}

We study Spatial Logistic Gaussian Process (SLGP) models for non-parametric estimation of probability density fields using scattered samples of heterogeneous sizes. 
SLGPs are examined from the perspective of random measures and their densities, investigating the relationships between SLGPs and underlying processes.
Our inquiries are motivated by SLGP's abilities in delivering probabilistic predictions of conditional distributions at candidate points, allowing conditional simulations of probability densities, and jointly predicting multiple functionals of target distributions. 
 We demonstrate that SLGP models 
 exhibit joint Gaussianity of their log-increments, enabling us to establish theoretical results regarding spatial regularity. Additionally, we extend the notion of mean-square continuity to random measure fields and establish sufficient conditions on covariance kernels underlying SLGPs to ensure these models enjoy such regularity properties.
Finally, we propose an implementation using Random Fourier Features and showcase its applicability on synthetic examples and on temperature distributions at meteorological stations.

\keywords{
Logistic Gaussian Process \and Spatial statistics \and Non-parametric models \and Spatial regularity \and Random measures  \and Distributional learning \and Covariance kernels}
\end{abstract}

\section{Introduction}

\subsection{Motivations and context}

One of the central problems in statistics and stochastic modelling is to capture and encode the dependence of a random response on a set of variables $\xX = (x_1, ..., x_d)$ in a flexible manner. Rather than treating $\xX$ as a covariate in a regression framework, we take the perspective of spatial statistics, where $\xX $ serves as an indexing variable for the evolution of a stochastic process over a domain of interest. 
In particular, we focus on the problem of modelling and predicting fields of response distributions indexed by $\xX$, based on scattered samples of heterogeneous sizes. 

Modelling and predicting (conditional) distributions depending on covariates is known to particularly challenging when this dependence does not only concern the mean and/or the variance of the distributions, but other features can evolve, including for instance their shape, their uni-modal versus multi-modal nature, etc. We refer to \cite{kneib_distribreg_2023} for a review of distributional regression approaches. In our specific context of distribution field, Figure~\ref{fig:meteo_at_stations_with_data_noSLGP} illustrates temperature distributions in Switzerland indexed by latitude, longitude, and altitude.
While in this temperature field example, the sample size is homogeneous across space, one of the motivations underlying the presented approach is to allow for heterogeneous sample sizes \citep{gautier_goal-oriented_2021} 

\begin{figure}
	\centering
	\includegraphics[width=0.8\linewidth]{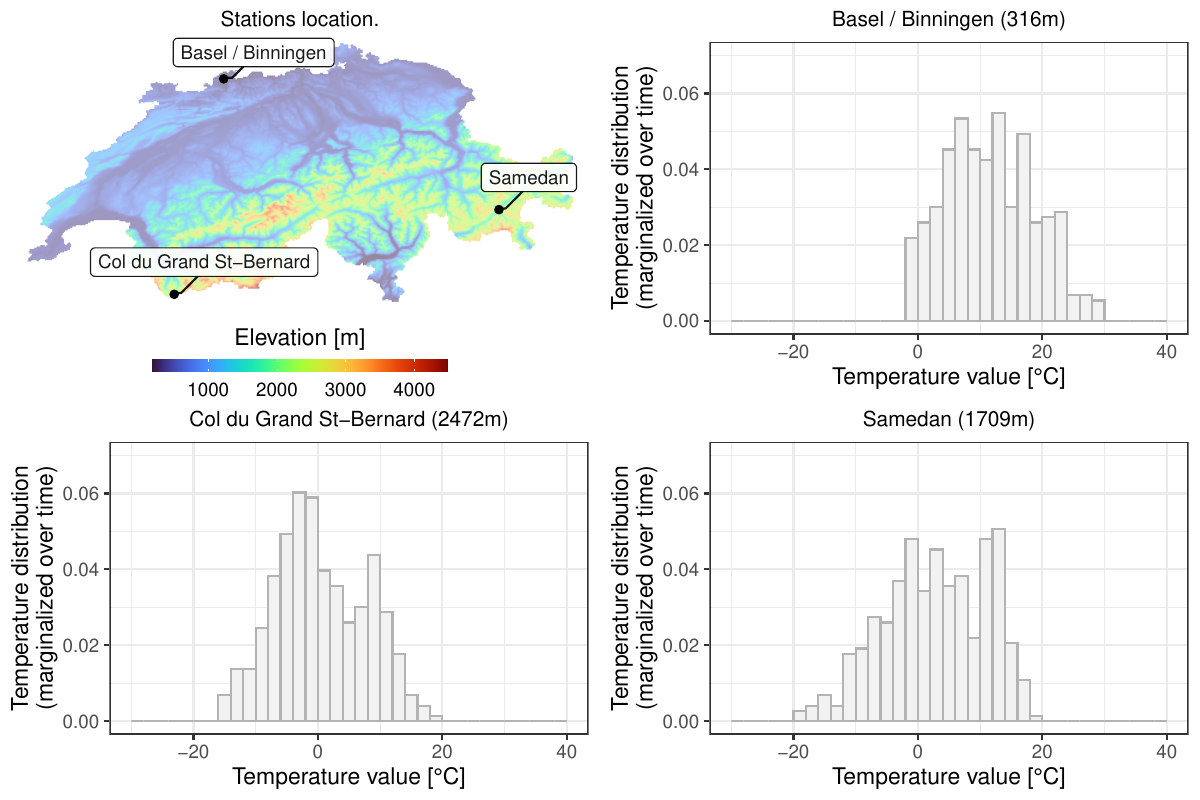}
	\caption{One example of probability distribution field: The daily mean temperatures in Switzerland indexed by latitude, longitude and altitude. We display the histogram of the available data for three meteorological stations (365 replications for each location). 
 }
	\label{fig:meteo_at_stations_with_data_noSLGP}
\end{figure}

Among the most notable approaches typically used in a frequentist framework to address (conditional) distributional modelling, one can cite finite mixture models \citep{rojas_conditional_2005} or kernel density estimation \citep{fan_estimation_1996,hall_methods_1999}. Kernel approaches usually require estimating the bandwidth which is done with cross-validation \citep{fan_crossvalidation_2004}, bootstrap \citep{hall_methods_1999} or other methods. Generalized lambda distributions have recently been used in \cite{zhu_emulation_2020, zhu_surrogate_2023} for flexible semi-parametric modelling of unimodal distributions depending on covariables. In addition to these approaches, distributional kriging has emerged as a method for spatial prediction of distributional data, offering a structured way to interpolate probability distributions in space \citep{aitchison_statistical_1982, egozcue_hilbert_2006, menafoglio_universal_2013, menafoglio_kriging_2014, talska_compositional_2018}. While distributional Kriging elegantly leverages functional Kriging, it requires fitted cumulative distribution functions at each location and does not account for the uncertainty due the fitting stage, making it hardly applicable for small and heterogeneous sample sizes. 
None of the aforementioned approaches simultaneously tackles all challenges of dealing with scattered samples with few or no replicates, handling distributions of various shapes and modalities, and providing uncertainty quantification.

The limitations of these methods, particularly in providing robust uncertainty quantification, prompt us to turn toward Bayesian approaches. Within this context, it is natural to put a prior on probability density functions and derive posterior distributions of such probability density functions given observed data. The most common class of models are infinite mixture models. Its popularity is partly due to the wide literature on algorithms for posterior sampling within a Markov Chain Monte Carlo framework \citep{jain_split-merge_2004, walker_sampling_2007, papaspiliopoulos_retrospective_2008} or fast approximation \citep{minka_family_2001}. 
Non-parametric approaches 
include generalizing stick breaking processes \citep{dunson_kernel_2008, dunson_bayesian_2007, chung_nonparametric_2009, griffin_order-based_2006}, multivariate transformation of a Beta distribution \citep{trippa_multivariate_2011} or transforming a Gaussian Process (GP) \citep{jara_class_2011, donner_efficient_2018, tokdar_bayesian_2010, gautier_goal-oriented_2021}. 
The Spatial Logistic Gaussian Process (SLGP) model, related to \cite{tokdar_bayesian_2010} and being at the center of the present contribution following up on \cite{gautier_goal-oriented_2021}, is itself a spatial generalization of the logistic Gaussian process model. 

The logistic Gaussian process for density estimation was established and studied in \cite{lenk_logistic_1988, lenk_towards_1991, leonard_density_1978} and is commonly introduced as a random probability density function obtained by applying a non-linear \textit{transformation} 
$\logt$ to a \textit{sufficiently well-behaved} GP $Z=\proc{Z}{t}{t \in \xI}$, resulting in 
\begin{equation}
\label{eq:informalLGP}
	\logt[Z](t)= \frac{e^{\proc{Z}{t}{}}}{\int_\xI e^{\proc{Z}{u}{}} \,d\lambda(u) } \text{ for all } t \in \xI
\end{equation}
Here and throughout the document, we consider 
a compact and convex response space $\xI \subset \xR$, and we denote by $\lambda$ the Lebesgue measure on $\xR$. We further assume that $\lambda(\xI) >0$.

For the Spatial Logistic Gaussian Process (SLGP), we will similarly build upon a \textit{well-behaved} GP $\proc{Z}{\xX, t}{(\xX, t) \in D \times \xI}$ (now indexed by a product set) and study the stochastic process obtained from applying the \textit{spatial logistic} transformation to $Z$ as follows:

\begin{equation}
\label{eq:informalSLGP}
	\slogt[Z](\xX, t)= \dfrac{e^{Z_{\xX, t}}}{\int_\xI e^{Z_{\xX, u}} \,d\lambda(u) } \text{ for all } (\xX, t) \in D \times \xI
\end{equation}
where $D \subset \xR^{\dimD}$ is a compact and convex index space with $\dimD \geq 1$. 

At any fixed $\xX$, $\slogt[Z](\xX, \cdot)$ returns a logistic Gaussian process, so that a SLGP can be seen as a field of logistic Gaussian processes. As motivated by theoretical and practical investigations conducted in the current paper and in \cite{gautier_modelling_2023}, the SLGP constitutes a worthwhile object of study. In particular, what the mathematical objects involved precisely are (in terms of random measures or densities, and fields thereof) calls for some careful analysis, underlying the first question that we will focus on in the present article:

\smallskip 

\noindent
\textbf{Question 1}: What (kind of stochastic models) are spatial logistic Gaussian processes?

\smallskip 

We revisit both logistic Gaussian process and SLGP models in terms of random measures and random measure fields, investigating in turn different notions of equivalence and indistinguishability between random measure fields. 

Also, since finite-dimensional distributions of Gaussian Processes are characterized by mean and covariance functions, it can be tempting to jump to the conclusion that the latter functions characterize SLGPs. As we show here, things are in fact not so simple, due in the first place to the fact that SLGPs require to control measurability and further properties of GPs beyond finite-dimensional distributions, but also because $\logt$ is invariant under translations of $Z$ by constants (resp. $\slogt$ is invariant under translations of $Z$ by functions of $\xX$). We go on with asking what characterizes SLGPs among fields of random measures with prescribed positivity properties:  

\smallskip 

\noindent
\textbf{Question 2}: Given a random field of positive densities, how to characterize that it is an SLGP and what can then be said about the underlying GP(s)?

\smallskip 

Thereby we keep a particular interest for the role of covariance kernels. We denote here and in the following by $k$ the covariance kernel of $Z$. 

In geostatistics and spatial statistics \citep{matheron_principles_1963, stein_interpolation_1999, cressie_statistics_1993}, quantifying the spatial regularity of scalar valued (Gaussian) processes such as $Z$ (with links to properties of $k$ in centered squared integrable cases) have been well studied and a wide literature is available. Similar approaches have been investigated in settings of function valued processes \citep{ramsay_functional_2004, henao_geostatistical_2009, nerini_cokriging_2010} but the main contributions in such cases are generally limited to stationary functional stochastic processes valued in $L^2$. Extensions to the distributional setting have also been proposed through embedding into an infinite-dimensional Hilbert Space using Aitchison geometry \citep{aitchison_statistical_1982, pawlowsky-glahn_compositional_2011, menafoglio_kriging_2014} and classical results on stationary functional processes. Here, we tackle notions of spatial regularity between (random) measures / probability densities that does not require Hilbert Space embedding. More specifically, we focus on:

\smallskip 

\noindent
\textbf{Question 3}: How to quantify spatial regularity in (logistic) random measure fields?

\smallskip 

For that we investigate generalizations of scalar-valued continuity notions (mean-square continuity and almost sure continuity) in the context of random measure fields, and especially of SLGPs. This leads us to our ultimate question, naturally following as an extension of results from the scalar valued case to (selected cases of) SLGPs:  

\smallskip 

\noindent
\textbf{Question 4}: To what extent is the spatial regularity of SLGPs driven by $k$?

\smallskip 

Sufficient conditions of mean-power continuity (with respect to the Total Variation and Hellinger distances, as well as to the squared log-ratio dissimilarity and to the Kullback-Leibler divergence) of SLGPs are established in terms of the covariance kernel underlying the considered SLGPs. Moreover, almost sure results are derived by building upon general results on Gaussian measures in Banach spaces. 

\medskip 

This article is structured in the following way: in Section \ref{sec:def}, we
mostly focus on Questions $1$ and $2$ by revisiting logistic Gaussian processes in terms of Random Probability Measures and extending this construction to SLGPs by introducing Random Probability Measure Fields (RPMFs).  In particular, Proposition~\ref{prop:charac_RMF}, we relate the indistinguishability of RPMFs to the indistinguishability of the increments of the underlying GPs. 

We also establish that, for any point $\xX$, a SLGP evaluated at $\xX$ can be interpreted as a representer of the Radon-Nikodym derivative of the corresponding RPMF at this same $\xX$. Moreover, we explore in Proposition~\ref{prop:SLGP_cty_case} the case where the transformed GP is a.s. continuous.

Throughout Section \ref{sec:cty}, we study the spatial regularity of the SLGP by relying on notions from spatial statistics and basic yet powerful results from Gaussian measure theory. Assuming that our SLGP is obtained by transforming $Z\sim \mathcal{GP}(0, k)$, we leverage the following condition on $k$:

\begin{condn}[\ref{con:suff_k}]
 There exist $C, \alpha_1, \alpha_2 >0$ such that for all $\xX, \xX' \in D, t, t' \in \xI$:
	\begin{equation*}
  k([\xX, t], [\xX, t]) + k([\xX', t'],[\xX', t']) - 2 k([\xX, t], [\xX', t']) \leq C \cdot \max(\Vert \xX - \xX' \Vert_\infty ^{\alpha_1}, \vert t - t' \vert ^{\alpha_2})
	\end{equation*}
\end{condn}

\noindent and establish the a.s. Hölder continuity of SLGPs as well as continuity rates:

\begin{thmn}[\ref{th:Expected_quadratic_cty}]
Consider the SLGP $Y$ induced by a measurable, separable, centred GP $Z$ with covariance kernel $k$ and assume that $k$ satisfies Condition~\ref{con:suff_k}.
	
	Then, for all $\gamma>0$ and $0<\delta < \gamma\alpha_1/2$ (for Equations~\ref{eq:cty_1bis}-\ref{eq:cty_0bis}, resp. $0<\delta < \gamma\alpha_1$ for Equations~\ref{eq:cty_2bis}-\ref{eq:cty_3bis}), there exists $K_{\gamma, \delta}>0$ such that for all $\xX, \xX' \in D^2$:
	\begin{align}
            \mathbb{E} \left[ d_{H}(\proc{Y}{\xX, \cdot}{}, \proc{Y}{\xX', \cdot}{})^\gamma \right] \leq K_{\gamma, \delta}  \Vert \xX - \xX' \Vert^{\gamma \alpha_1 /2 -\delta}_\infty \label{eq:cty_1bis}\\
            \mathbb{E} \left[ V(\proc{Y}{\xX, \cdot}{}, \proc{Y}{\xX', \cdot}{})^\gamma \right] \leq K_{\gamma, \delta}  \Vert \xX - \xX' \Vert^{\gamma \alpha_1/2 -\delta}_\infty \label{eq:cty_0bis}\\
            \mathbb{E} \left[ KL(\proc{Y}{\xX, \cdot}{}, \proc{Y}{\xX', \cdot}{})^\gamma \right] \leq K_{\gamma, \delta}  \Vert \xX - \xX' \Vert^{\gamma \alpha_1 -\delta}_\infty \label{eq:cty_2bis}   \\
            \mathbb{E} \left[ d_{TV}(\proc{Y}{\xX, \cdot}{}, \proc{Y}{\xX', \cdot}{})^\gamma \right] \leq K_{\gamma, \delta}  \Vert \xX - \xX' \Vert^{\gamma \alpha_1 -\delta}_\infty \label{eq:cty_3bis} 
        \end{align}
\end{thmn}

Some results over analytical test distribution fields and a test case based on meteorological data are presented in Section \ref{sec:app}. 
From the computational side, we introduce a Markov Chain Monte Carlo algorithmic approach for conditional density estimation relying on Random Fourier Features approximation \citep{rahimi_random_2008, rahimi_weighted_2009}, and apply it to a meteorological data set. 

We also include detailed proofs in Appendix \ref{app:fullproofs}, and supplementary material \citep{gautier_supplementary_2023} pertaining to definitions and properties of the notion of consistency as well as to posterior consistency  (Section~4
of the supplementary material) and details on the implementation of the density field estimation (Section~5 
of the supplementary material).

\subsection{Further notations and working assumptions}
Throughout the article, we denote by $(\Omega, \mathcal{F}, P)$ the ambient probability space and by $\sigFieldM$ the Borel $\sigma$-algebra induced by the Euclidean metric on $\xI$.
For a set $S$ (here $\xI$ or $D\times \xI$), we further denote by $\mathcal{C}^0(S)$, $\pdf(S)$, $\pospdf(S)$ the sets of continuous real functions, Probability Density Functions (PDFs) and positive PDFs on $S$, respectively.  
Finally, we denote by $\indpdf(D; \xI)$ the set of fields of PDFs on $\xI$ indexed by $D$, and by $\indpospdf(D; \xI)$ its counterpart featuring positive PDFs.
Note that for clarity, $\xX \in D$ (and variations thereof) always plays the role of the spatial index, whereas $t \in \xI$ (and variations thereof) refers to the response. 

Moreover, to alleviate technical difficulties, we will always assume that the Random Fields (RF) considered are measurable, as well as separable whenever almost sure continuity is mentioned.

\section{Logistic Gaussian random measures and measure fields}
\label{sec:def} 
Generative approaches to sample-based density estimation build upon probabilistic models for the unknown densities. 
A convenient option to devise such probabilistic models over the set $\pdf(\xI)$ consists in re-normalising non-negative random functions that are almost surely integrable.

When the random density is obtained by exponentiation and normalization of a Gaussian process, the resulting process is called logistic Gaussian process. 

In this first section, we start by giving an historical perspective on logistic Gaussian process models underlying this work. Following this, we address the first two questions concerning the stochastic nature and distributional characteristics of (Spatial) logistic Gaussian Process.
We directly take the broader perspective of spatial logistic Gaussian processes (SLGPs) for these inquiries, as the LGP case naturally arises as a specific instance within this more general perspective.

\subsection{The logistic Gaussian process for density estimation}

Recall that we informally introduced the logistic Gaussian process in Equation~\ref{eq:informalLGP} as being obtained through exponentiation and normalisation of a \textit{well-behaved} GP $Z$:
\begin{equation*}
	\logt[Z](t)= \dfrac{e^{\proc{Z}{t}{}}}{\int_\xI e^{\proc{Z}{u}{}} \,d\lambda(u) } \text{ for all } t \in \xI
\end{equation*}

These models intend to provide a flexible prior over positive density functions, where the smoothness of the generated densities is directly inherited from the GP's smoothness.

In the literature, various assumptions and theoretical settings have been proposed that (often, implicitly) specify what \textit{well-behaved} refers to and in what sense the colloquial definition above is meant. We present a concise review of a few papers among the ones we deem to be most representative on the topic and refer the reader 
to section~2 
of the supplementary material \cite{gautier_supplementary_2023} for more details. 

What we find noticeable is that working assumptions fluctuate between different contributions, and there does not seem to be a consensus on the most appropriate set of hypotheses. In particular, the choice between having $Z$ enjoy properties almost-surely (e.g. continuity) or surely (e.g. measurability) is far from being straightforward.

In the seminal paper \cite{leonard_density_1978}, the authors consider a.s. continuous GPs with exponential covariance kernel. Later, the author of \cite{lenk_logistic_1988, lenk_towards_1991} claims that logistic Gaussian processes should be seen as positive-valued random functions integrating to 1 but fail to provide an explicit construction of the corresponding measure space.  The construction in \cite{tokdar_posterior_2007} and \cite{tokdar_towards_2007} requires considering a separable GP $Z$ that is exponentially integrable almost surely, stating that the logistic Gaussian process thus a.s. takes values in $\pdf(\xI)$. Meanwhile, the authors of \cite{van_der_vaart_rates_2008} work with GPs whose sample paths are (surely) bounded functions.\\

In an attempt to establish transparent mathematical foundations and identify a smallest convenient set of working hypotheses, we do not 
view logistic Gaussian processes as random functions satisfying constraints (namely: non-negativity, and integrating to 1), 
but rather introduce them through the scope of random measures. We rely on the definitions from \cite{kallenberg_random_2017}, that are recalled in 
Section 1.2
of the supplementary material \citep{gautier_supplementary_2023} 
and that allow us to work with random probability measures. This approach facilitates working within the measurability structure of random probability measures and identifying a set of hypotheses on $Z$ for the corresponding transformed process to be well-defined.


\subsection{On spatial logistic Gaussian process models and associated random measure fields}
\label{sec:def:subsec:slgp}


Building upon the work of \cite{pati_posterior_2013}, we study the spatial extension of the Logistic Gaussian Process.

In this context, we call a measurable GP $\proc{Z}{\xX, t}{(\xX, t) \in D \times \xI}$ \emph{exponentially measurable alongside $\xI$} if $\int_\xI e^{\proc{Z}{\xX, u}{}(\omega)} \,d\lambda(u) < \infty$ for any $(\xX, \omega) \in D \times \Omega$. 

We begin by exploring the spatial extension of the logistic density transformation:

\begin{definition}[Spatial logistic density transformation]
	\label{def:spat_log_transfor}
	The spatial logistic density transformation $\slogt$ is defined over the set of measurable functions $f: D\times \xI \rightarrow \xR$ such that for all $\xX \in D$, $\int_\xI e^{f(\xX, u)} \,d\lambda(u) < \infty$, by:
	\begin{equation}
		\slogt[f](\xX, t):= \dfrac{e^{f(\xX, t)}}{\int_\xI e^{f(\xX, u)} \,d\lambda(u) } \ \ \ \text{ for all } (\xX, t) \in D \times \xI
	\end{equation}
	hence being a mapping between functions that are exponentially integrable alongside $\xI$ and $\indpospdf(D; \xI)$.
\end{definition}

Having (re)introduced the Spatial logistic density transformation, we now study its connections to (some classes of) Random Measure Fields, providing a broader foundation for understanding spatially dependent probability distributions. This perspective provides a structured way to capture underlying properties and relationships between spatially varying probability distributions.




\subsubsection{RPMFs induced by a random field: definition and characterisation}

In this subsection, we formally define a specific class of RPMFs induced by random fields through spatial transformations. We then explore their relationships with SLGPs and provide their characterisations. To assist the reader in understanding these connections, we summarise these relationships in Figure~\ref{fig:math_objects}, which serves as a visual guide for navigating the key concepts discussed hereafter.

\begin{definition}[Logistic Random Probability Measure Field - LRPMF]
	\label{def:SLGPM} 
        For $\proc{Z}{\xX, t}{(\xX, t) \in D \times \xI}$, a scalar-valued random field that is exponentially integrable alongside $\xI$, we define by:
	\begin{equation}
	\label{eq:defSLGPM}
		\SLGPM{\xX}{B} = \int_{B} \slogt[ \proc{Z}{}{} ](\xX, u) \,d\lambda(u) = \frac{\int_B e^{\proc{Z}{\xX, u}{}} \,d\lambda(u)}{\int_\xI e^{\proc{Z}{\xX, u}{}} \,d\lambda(u)} \quad (\xX \in D, \ B \in \sigFieldM)
	\end{equation}
        a RPMF that we call \emph{Logistic Random Probability Measure Field} induced by $Z$. We also use the notation $ \SLGPM{}{} = \text{LRPMF}(Z)$.
\end{definition}

\begin{definition}[Spatial Logistic Process - SLP]
\label{def:formalSLP}
    Let $\proc{Z}{\xX, t}{(\xX, t) \in D \times \xI}$ be a real-valued random field that is exponentially integrable alongside $\xI$, then:
    \begin{equation}
		\label{eq:defSLPformal}
  \slogt[Z](\xX, t)= \dfrac{e^{Z_{\xX, t}}}{\int_\xI e^{Z_{\xX, u}} \,d\lambda(u) } \text{ for all } (\xX, t) \in D \times \xI
	\end{equation}
    defines a $\pospdf(\xI)$-valued random process. It follows that for $ \SLGPM{}{} = \text{LRPMF}(Z)$, for any $\xX \in D$, $\slogt[Z](\xX, \cdot)$ is a representer of $\dfrac{d \SLGPM{\xX}{} }{d \lambda}$, the Radon–Nikodym derivative of $\SLGPM{\xX}{} $. We denote this process by $\slogt[Z]=\left(\slogt[Z](\xX, \cdot)\right)_{\xX \in D}$ and refer to it as Spatial Logistic Process (SLP).
\end{definition}

While it is tempting to characterise a LRPMF by its underlying random field, it is hopeless. In fact, different random fields may yield the same LRPMF. 

\begin{remark}
	\label{remark:non_unicity_GPtoSLGP}
	Let us consider two RFs $\proc{Z}{\xX, t}{(\xX, t) \in D \times \xI}$ and $\proc{R}{\xX}{\xX \in D}$ defined on the same probability space, and assume that $Z$ is exponentially integrable alongside $\xI$. Then, $\slogt[Z]$ and $\slogt[Z+R]$ are equal. Indeed, for any $(\xX, t) \in D \times \xI$ and $\omega \in \Omega$ we have:
	\begin{equation*}
		\frac{e^{\proc{Z}{\xX, t}{}(\omega)} }{\int_\xI e^{\proc{Z}{\xX, u}{}(\omega)} \,d\lambda(u)} = 
		\frac{e^{\left[ \proc{Z}{\xX, t}{}+\proc{R}{\xX}{}\right](\omega)} }{\int_\xI e^{\left[ \proc{Z}{\xX, u}{}+\proc{R}{\xX}{}\right](\omega)} \,d\lambda(u)} 
	\end{equation*}
    It follows that $\text{LRPMF}(Z)$ and $\text{LRPMF}(Z+R)$ are also equal.
\end{remark}

The arising questions that we will try to address through the rest of this section is: how to characterise the random measure fields that can be obtained through Equation~\ref{eq:defSLGPM}, and can we give sufficient conditions on measurable and exponentially integrable RFs for them to yield the same LRPMF?

\begin{figure}[H]
\centering
\begin{tikzpicture}[x=\linewidth/400, y=0.75pt,yscale=-1,xscale=1.01]
\draw  [rectangle, draw, rounded corners] (5, 35) -- (195, 35) -- (195, 290) -- (5,290) -- cycle ;
\draw  [rectangle, draw, rounded corners, color=blue!80!black] (10, 105) -- (190, 105) -- (190, 285) -- (10,285) -- cycle ;

\draw  [rectangle, draw, rounded corners, color=blue!80!black] (205, 65) -- (395, 65) -- (395, 290) -- (205, 290) -- cycle ;
\draw  [rectangle, draw, rounded corners] (210, 285) -- (390, 285) -- (390, 185) -- (210, 185) -- cycle ;

\draw  [color=green!60!black] (155,195) -- (107,195) -- (107, 215) -- (93,215) -- (93,237) -- (155,237) -- cycle ;

\draw  [color=green!60!black] (285,142) -- (340,142) -- (340,183) -- (285,183) -- cycle ;
\draw [line width=1em,{Triangle Cap []}-{Triangle Cap []}][color=white]   (100, 312) .. controls (175, 330) and (225, 330) .. (300, 317) ;
\draw [{Stealth}-{Stealth}][color=blue!80!black]   (100, 285) .. controls (175, 315) and (225, 315) .. (300, 290) ;

\draw (5,40) node [anchor=north west, text width=7.5cm, align=left] {\textbf{Random Probability Measure Field} (RPMF)};
\draw (5,70) node [anchor=north west, text width=7.5cm, align=left] {A collection $(\SLGPM{\xX}{})_{\xX \in D}$ of random probability measures indexed by $\xX$.};
\draw (10,110) node [anchor=north west, text width=7.5cm, align=left]{\textbf{Logistic Random Probability Measure Field} (LRPMF)};
\draw (10,140) node [anchor=north west, text width=7.6cm, align=left]{
  A RPMF $\SLGPM{}{}$ for which there exist a real valued random field  $\proc{Z}{\xX, t}{(\xX, t) \in D \times \xI}$ with, for all $\xX \in D, \ B \in \sigFieldM$:
  \begin{equation*}
		\SLGPM{\xX}{B} = \frac{\int_B e^{\proc{Z}{\xX, u}{}} \,d\lambda(u)}{\int_\xI e^{\proc{Z}{\xX, u}{}} \,d\lambda(u)} \end{equation*}
        \vspace{2pt}
        
        \color{black!70}Characterized by increments of $Z$ at fixed $\xX$: $\proc{\Delta Z}{\xX, t, t'}{}:= \proc{ Z}{\xX, t}{} - \proc{ Z}{\xX, t'}{}$};
  
\draw (205,70) node [anchor=north west, text width=7.5cm, align=center]   [align=left] {\textbf{Spatial Logistic Process} (SLP)};
\draw (205,87) node [anchor=north west, text width=7.7cm, align=center]   [align=left] {A process $Y=\proc{Y}{\xX, t}{(\xX, t) \in D \times \xI}$ for which there exist a real valued random field  $\proc{Z}{\xX, t}{(\xX, t) \in D \times \xI}$ with for all $\xX \in D, t \in \xI$:
\begin{equation*}
  \proc{Y}{\xX, t}{}= \dfrac{e^{Z_{\xX, t}}}{\int_\xI e^{Z_{\xX, u}} \,d\lambda(u) }
	\end{equation*}
    };
    
\draw (210,190) node [anchor=north west, text width=7.5cm, align=center]   [align=left] {\textbf{Spatial Logistic Gaussian Process} (SLGP)};
\draw (210,220) node [anchor=north west, text width=7.5cm, align=center]   [align=left] {A SLP $Y$ where $Z$ is a GP.
\vspace{3.5pt}
        
\color{black!70}Characterized by the Gaussianity of $\log Y_{\xX, t} - \log Y_{\xX, t'}$};

\draw (200,292) node [anchor=north, text width=5.5cm, align=center, color=blue!80!black] {\contour{white}{SLP is a Radon-Nikodym} \contour{white}{derivative of LRPMF.}};

\draw (15,97.5) node [anchor=north west, text width=5.5cm, align=left, color=blue!80!black] {\scriptsize \contour{white}{Particular case of RPMF}};

\draw (215,177) node [anchor=north west, text width=5.5cm, align=left, color=black] {\scriptsize \contour{white}{Particular case of SLP}};
\end{tikzpicture}
\vspace{-0.6cm}
\caption{Summarising the nature and relationships between the mathematical objects considered in this section}
\label{fig:math_objects}
\end{figure}

\vspace{-0.2cm}

To answer these questions, we must define what we mean by ``the same LRMPF'', as there are various notions of coincidence between RFs (and by extension, RPMFs). In this discussion, we will primarily focus on the concept of indistinguishability and refer the reader to Remark D
in the supplementary material for a reminder on this notion.

\begin{proposition}[Condition for the indistinguishability of LRPMF and SLPs]
	\label{prop:charac_RMF}
	Let $ Z := \proc{Z}{\xX, t}{(\xX, t) \in D \times \xI}$ and $\Tilde Z := \proc{\Tilde Z}{\xX, t}{(\xX, t) \in D \times \xI}$ be two RFs that are exponentially integrable alongside $\xI$, and for all $(\xX, t, t') \in D \times \xI^2$, let $ \proc{\Delta Z}{\xX, t, t'}{(\xX, t, t') \in D \times \xI^2}$ and $ \proc{\Delta \Tilde Z}{\xX, t, t'}{(\xX, t, t') \in D \times \xI^2}$ be the associated increment processes along $\xX$, where $ \proc{\Delta Z}{\xX, t, t'}{}:= \proc{ Z}{\xX, t}{} - \proc{ Z}{\xX, t'}{}$ (respectively $ \proc{\Delta \Tilde Z}{\xX, t, t'}{}:= \proc{\Tilde Z}{\xX, t}{} - \proc{\Tilde Z}{\xX, t'}{}$) for any $ (\xX, t, t') \in D \times \xI^2$. Then,
	\begin{enumerate}
		\item[\hspace{10pt}$(1)$] $ \proc{\Delta Z}{\xX, t, t'}{(\xX, t, t') \in D \times \xI^2}$ is indistinguishable from  $ \proc{\Delta \Tilde Z}{\xX, t, t'}{(\xX, t, t') \in D \times \xI^2}$.
		\item[$\Leftrightarrow (2)$] $\slogt[Z]$ is indistinguishable from $\slogt[\tilde Z]$.
        \item[$\Rightarrow (3)$] $\text{LRPMF}( Z)$ is indistinguishable from $\text{LRPMF}(\tilde Z)$.
	\end{enumerate}

 Additionally, if $Z$ and $\tilde Z$ are almost surely continuous, $(3) \Rightarrow (2)$.
\end{proposition}
\begin{proof}
    Let us consider two such RFs $Z$ and $\tilde Z$. Assuming the two increment fields are indistinguishable, for an arbitrary $t_0 \in \xI$, both $\Delta Z_{\cdot \cdot t_0}$  and $\Delta \tilde Z_{\cdot \cdot t_0}$ are indistinguishable RFs that are exponentially integrable alongside $\xI$. Note that $\slogt[\Delta \tilde Z_{\boldsymbol \cdot \boldsymbol \cdot t_0}]=\slogt[\tilde Z]$, and therefore:
  \begin{align*}
    1 &= P\left[ \slogt[\Delta \tilde Z_{\boldsymbol \cdot \boldsymbol \cdot t_0}](\xX, t) = \slogt[ \Delta Z_{\boldsymbol \cdot \boldsymbol \cdot t_0}](\xX, t) \ \forall (\xX, t) \in D \times \xI \right]\\
    &= P\left[ \slogt[\tilde Z](\xX, t) = \slogt[ Z ](\xX, t) \ \forall (\xX, t) \in D \times \xI \right]
\end{align*}  
It follows that $(1) \Rightarrow (2)$. Moreover, it also follows that:
\begin{align*}
    1 &= P\left[ \slogt[\tilde Z](\xX, t) = \slogt[ Z ](\xX, t) \ \forall (\xX, t) \in D \times \xI \right]\\
    &= P\left[ \int_B \slogt[\tilde Z](\xX, u) \,d\lambda(u) = \int_B \slogt[ Z ](\xX, u) \,d\lambda(u), \ \ \forall (\xX, B) \in D \times \sigFieldM \right]
\end{align*}  
Which proves that $(2) \Rightarrow (3)$.

Conversely, for $(2)\Rightarrow(1)$, let us assume that $Y=\slogt[Z]$ is indistinguishable from $\tilde Y=\slogt[\tilde Z]$. By SLP's construction, we can consider $ \log \proc{Y}{\xX, t}{} =  \proc{Z}{\xX, t}{} - \log \int_\xI e^{ \proc{Z}{\xX, u}{} }\,d\lambda(u)$ (resp. $ \log \proc{\tilde Y}{\xX, t}{} =  \proc{\tilde Z}{\xX, t}{} - \log \int_\xI e^{ \proc{\tilde Z}{\xX, u}{} }\,d\lambda(u)$), and:
\begin{align*}
    1 &= P\left[ \log \proc{Y}{\xX, t}{} = \log \proc{\tilde Y}{\xX, t}{} \ \forall (\xX, t) \in D \times \xI \right]\\
    &= P\left[ \log \proc{Y}{\xX, t}{} - \log \proc{Y}{\xX, t'}{} = \log \proc{\tilde Y}{\xX, t}{}- \log \proc{\tilde Y}{\xX, t'}{} \ \forall (\xX, t, t') \in D \times \xI^2 \right]\\
    &= P\left[ \proc{Z}{\xX, t}{} - \proc{Z}{\xX, t'}{} = \proc{\tilde Z}{\xX, t}{}- \proc{\tilde Z}{\xX, t'}{} \ \forall (\xX, t, t') \in D \times \xI^2 \right]
\end{align*}  
which is, indeed, proving $(2)\Rightarrow(1)$.

Finally, let us assume that $(3)$ holds and that $\SLGPM{\xX}{}=\text{LRPMF}(Z)$ is indistinguishable from $\tilde \Xi_{\xX}=\text{LRPMF}(\tilde Z)$.
By indistinguishability:
\begin{align*}
     &P\left[ \SLGPM{\xX}{}(B) = \Tilde \Xi_\xX(B) ,\  \forall \xX \in D, \forall B \in \mathcal B(\xI) \right] = 1\\
     &\Leftrightarrow P\left[ \proc{ Y}{\xX, t}{} = \proc{ \tilde Y}{\xX, t}{}  ,\ \text{ for }\lambda\text{-almost every } t\in \xI, \text{ for all } \xX \in D \right] = 1
\end{align*}
Under the general setting, this is not enough to prove that $(3)\Rightarrow(2)$. 

However, assuming that both $Z$ and $\tilde Z$ are a.s. continuous, we deduce that so are $Y$ and $\tilde Y$. This allows for going from almost sure equality $\lambda$-almost everywhere to almost sure equality everywhere, and therefore:
\begin{align*}
    P \left[ \proc{ Y}{\xX, t}{} = \proc{ \tilde Y}{\xX, t}{}  ,\ \text{ for all }t\in \xI, \xX \in D \right] = 1
\end{align*}
\end{proof}

\begin{remark}[Indistinguishability compared to others notions of coincidence between RPMF]
\label{remark:RMF_indist_benefits}
    In Proposition~\ref{prop:charac_RMF}, we worked with the indistinguishability of random measure fields.
    Although one could consider other types of equality between RPMF, such as the \emph{equality up to a modification}: 
    \begin{equation}
    \label{eq:version_RMF}
        P\left[ \SLGPM{\xX}{} = \Tilde \Xi_\xX \right] = 1 \quad \forall\  \xX  \in D 
    \end{equation}
    we have found that indistinguishability is the most appropriate notion in this context, as it naturally connects the indistinguishability of LRPMFs with that of the underlying fields of increments.
\end{remark}



From this characterisation, it appears that indistinguishability of SLPs or LRPMFs is driven by the increments of the transformed RF. It also appears that almost sure continuity is a practical assumption to alleviate technical difficulties. However, it also highlights how general our construction is, indeed:

\begin{lemma}
    Let us consider a RPMF $(\SLGPM{\xX}{})_{\xX \in D}$, if there exists a measurable process $Y=\proc{Y}{\xX, t}{(\xX, t) \in D \times \xI}$ with:
    \begin{equation}
        P \left[ \proc{Y}{\xX, \cdot}{} \text{ is a positive representer of } \dfrac{d \SLGPM{\xX}{} }{d \lambda} \text{ for all } \xX \in D \right]=1
    \end{equation}
   then there exist a RF $Z=\proc{Z}{\xX, t}{(\xX, t) \in D \times \xI}$ exponentially integrable alongside $\xI$ such that $Y$ is indistinguishable from $\slogt[Z]$, and $\Xi$ is indistinguishable from LRPMF$(Z)$.
\end{lemma}

\begin{proof}
    Let us consider such a $\proc{Y}{\xX, t}{(\xX, t) \in D \times \xI}$ and denote by $N$ the $P$-null set where $Y$ is not the positive representer of the Radon-Nikodym derivative of $\SLGPM{\xX}{}$ for all $\xX \in D$.
    We define $Z=\proc{Z}{\xX, t}{(\xX, t) \in D \times \xI}$ by:
    \begin{equation}
        \proc{Z}{\xX, t}{}(\omega) := \left\{ \begin{array}{ll}
             \log  \proc{Y}{\xX, t}{}(\omega)  &  \text{ if } \omega \in \Omega\setminus N\\
             0 & \text{otherwise}
        \end{array}\right.
    \end{equation}
    By construction, $Z$ is a RF that is exponentially integrable alongside $\xI$, and $\slogt[Z]$ is indistinguishable from $Y$. It follows from Proposition~\ref{prop:charac_RMF} that LRPMF$(Z)$ is indistinguishable from $\Xi$.
\end{proof}

Therefore, LRPMFs are quite general objects, capable of modeling a wide class of RPMFs. In practice, we typically construct our models by specifying a $Z$, and then transforming it to obtain the corresponding SLPs or LRPMFs.  In the following section, we will focus specifically on LRPMFs induced by Gaussian Processes.

\subsubsection{LRPMFs and SLPs induced by a GP}
\label{subsubsec:GLRPMF}
We now characterise which SLPs are obtained by transforming a GP.

\begin{proposition}[Characterizing SLPs obtained by transforming a GP] \ \
\label{prop:LRPMFwithGP}
   For a measurable RF $\proc{Z}{\xX, t}{(\xX, t) \in D \times \xI}$ that is exponentially integrable alongside $\xI$, the following are equivalent:
   \begin{enumerate}
       \item[(1)] There exist a measurable GP $\proc{\tilde Z}{\xX, t}{(\xX, t) \in D \times \xI}$ that is exponentially integrable such that $\slogt[Z]=\slogt[\tilde Z]$
       \item[(2)] $\left( \proc{ Z}{\xX, t}{} - \proc{ Z}{\xX, t'}{}\right)_{(\xX, t, t') \in D \times \xI^2}$, is a GP.
   \end{enumerate}
\end{proposition}
\begin{proof}
If $(1)$ holds, then $\proc{ Z}{\xX, t}{} - \proc{ Z}{\xX, t'}{} = \proc{\tilde Z}{\xX, t}{} - \proc{\tilde Z}{\xX, t'}{}$ for all $(\xX, t)\in D \times \xI$. Since $\tilde Z$ is a GP, so is its increment field, and so is $Z$'s one.

Conversely, consider $Z$ as in  $(2)$. For an arbitrary $t_0  \in \xI$, let us define $\tilde Z:= \proc{Z}{\xX, t}{} - \proc{Z}{\xX, t_0}{}$. The process $\tilde Z$ is a GP on $D\times \xI$, and for any $(\xX, t) \in \mathcal D \times \xI$:
    \begin{equation*}
        \slogt[\tilde Z](\xX, t)  = \frac{ e^{\proc{ Z }{\xX, t}{}-\proc{ Z}{\xX, t_0}{}} }{\int_\xI e^{\proc{ Z}{\xX, u}{}-\proc{ Z}{\xX, t_0}{}} \,d\lambda(u)} =  \frac{ e^{\proc{Z}{\xX, t}{}} }{\int_\xI e^{\proc{ Z}{\xX, u}{}} \,d\lambda(u)}  =\slogt[Z](\xX, t)  
    \end{equation*}

    \vspace{-12pt}
\end{proof}

This observation suggests that SLPs obtained by transforming a GP are not characterized by a GP on $D \times \xI$ but rather by an increment (Gaussian) process on  $D \times \xI^2$ with sufficient regularity.

To streamline notation and connect our work to previous contributions from other authors, from now on we will refer to SLPs obtained by transforming a GP as \emph{Spatial Logistic Gaussian Processes (SLGPs)}.

Throughout our inquiries and studies, we noted that in practice, SLGPs benefit significantly from continuity assumptions, which facilitate their characterisation and parametrisation. Specifically, continuity streamlines the connection between an SLGP and its LRPMF, making the SLGP the sole continuous representative of the Radon-Nikodym derivative. From it, we derive another characterisation of LRPMFs obtained by transforming a.s. continuous GPs.

\begin{proposition}[Underlying increment mean and covariance]
   \label{prop:inc_mean_cov}
    
    For every LRPMF $(\SLGPM{\xX}{})_{\xX\in D}$ (resp. SLGP $Y=\proc{Y}{\xX, t}{(\xX, t) \in D \times \xI}$) induced by an almost surely continuous GP $Z=\proc{Z}{\xX, t}{(\xX, t) \in D \times \xI}$, there exist a unique mean function $\mincrement$ and a unique covariance kernel $\kincrement$:
    \begin{align}
        \mincrement:& D \times \xI^2 \rightarrow \xR\\
        \kincrement:& \left(( D \times \xI^2) \times ( D \times \xI^2)\right) \rightarrow \xR
    \end{align}
    that characterise all the LRPMFs indistinguishable from $\Xi$ (resp. the SLGPs indinstinguishable from $Y$). We call them the mean and covariance underlying the LRPMF (resp. the SLGP).	
\end{proposition}

\begin{proof}
    Combining Propositions~\ref{prop:charac_RMF} and \ref{prop:LRPMFwithGP} emphasizes that the process that drives $\Xi$ and $Y$'s behaviours is $
    \left( Z_{\xX, t} - Z_{\xX, t'} \right)_{(\xX, t, t') \in D \times \xI^2}
    $. It is an a.s. continuous GP, with $\mincrement$ and $\kincrement$ being its mean function and covariance function. The indistinguishability of continuous GPs is driven by these functions, which ensures that $\mincrement$ and $\kincrement$ characterise $\Xi$ (resp. $Y$). 
    For proofs of the latter statement, we refer the reader to \cite{azais_level_2009} Ch. 1, Sec. 4, Prop. 1.9 (in dimension 1), and \cite{scheuerer_comparison_2009} Ch. 5 Sec. 2 Lemma 5.2.8. (generalisation to greater dimension).
\end{proof}

\begin{proposition}[SLGP induced by an a.s. continuous GP]
    \label{prop:SLGP_cty_case}
   Let us consider a GP $Z=\proc{Z}{\xX, t}{(\xX, t) \in D \times \xI}$ that is exponentially integrable alongside $\xI$ and almost-surely continuous in $t$. Then, for any $\xX \in D$, $\slogt[Z](\xX, \cdot)$ is almost surely the continuous representer of $\dfrac{d \SLGPM{\xX }{}}{d \lambda}$, where $\SLGPM{}{} = \text{LRPMF}(Z)$.
   
\end{proposition}

A direct consequence of Proposition~\ref{prop:SLGP_cty_case}, is that whenever these assumptions on $Z$ are fulfilled, we can refer to the corresponding SLGP as being almost surely a probability density functions field. Moreover, it is possible to characterize a.s. continuous GPs that yield indistinguishable SLGPs directly through their kernels and means. Notably, the kernel and mean not only define these objects but also intrinsically determine the continuity of the process, which is a significant aspect of their characterization.

\begin{proposition}[Increment moments of GPs underlying a SLGP]
    \label{prop:underlying_mean_ker_SLGP}

We consider $Z=\proc{Z}{\xX, t}{(\xX, t) \in D \times \xI}$, a GP with mean $m$ and covariance $k$ that is exponentially integrable alongside $\xI$. We denote $\mincrement$ and $\kincrement$ the underlying mean and covariance of the SLGP (or a LRPMF) generated by $Z$. The following relations are satisfied for all $(\xX, \xX') \in D^2$, $(t_1, t_2, t'_1, t'_2) \in \xI^4$:
 	\begin{equation}
  	\label{eq:m_can_and_m}
		\mincrement (\xX, t_1, t_2) =  m(\xX, t_1)- m(\xX, t_2)
	\end{equation}
	\begin{equation}
	    \label{eq:k_can_and_k}
		\kincrement ([\xX, t_1, t_2], [\xX', t'_1, t'_2]) = \begin{array}{c}
			k([\xX, t_1], [\xX', t'_1]) 
			+ k([\xX, t_2], [\xX', t'_2]) \\
			- k([\xX, t_1], [\xX', t'_2]) 
			- k([\xX, t_2], [\xX', t'_1])
		\end{array}
	\end{equation}
\end{proposition}

This last property is central, as we have already noted that in practice, defining an SLGP is typically more straightforward by specifying $Z$. It also enables us to identify the equivalence classes of GPs that will give the same SLGPs.

In the remainder of this document, we will explore the spatial regularity of the SLGP and propose an implementation within a Bayesian framework. Additionally, we address the posterior consistency of this model in the supplementary material.

\section{Continuity modes for (logistic Gaussian) random measure fields}
\label{sec:cty}
Our object of interest here is a random measure field. A natural question when working with spatial objects is to quantify the spatial regularity. In our context, this involves measuring the similarity between two probability measures $\SLGPM{\xX}{}$, $\SLGPM{\xX'}{}$ depending on the closeness of their respective indices $\xX, \xX'$.

This investigation requires distances (or dissimilarity measures) between both measures and locations, and we will consider different ones. To compare two measures, we will consider  appropriate distances.
For locations, we will consider the sup norm over $D$ as well as the canonical distance associated to the covariance kernel of the Gaussian random field of increments.

Within this context, we focus on two notions of regularity: (i) the almost sure continuity of the SLGP, and (ii) a generalization of mean-squared continuity from the scalar-valued case to the measure-valued setting. Regarding the latter, we will prove results of the following form: given  a dissimilarity $\mathfrak{D}$ between probability measures and a SLGP $Y=\proc{Y}{\xX, t}{(\xX, t) \in D \times \xI}$, under sufficient regularity of the covariance kernel $\kincrement$ underlying $Y$, the following holds: 
\begin{equation}
	\label{eq:lim_tend_to_0}
	\lim_{\xX' \to\xX} \mathbb{E}\left[ \mathfrak{D} \left( \proc{Y}{\xX, \cdot}{}, \proc{Y}{\xX', \cdot}{} \right)\right] = 0.
\end{equation}
We will also establish bounds on the convergence rate.
In this work, we focus on the following dissimilarity measures for $\mathfrak{D}$: $d_H$ the Hellinger distance, $d_{TV}$ the Total Variation distance, $KL$ the Kullback-Leibler divergence, or $V$ a squared log-ratio dissimilarity: $V(f_1, f_2) := \int_\xI \left( \log \frac{ f_1(u)}{f_2(u)}\right)^2 \,du$. The choice of these specific dissimilarity measures is motivated by the following Lemma.

\begin{lemma}[Bounds on distances between measures]
	\label{pty:bound_logtransf_sup}
	There exists constants $C_{KL},  C_{V}, C_{TV} > 0$ such that for $f_1$ and $f_2$ two positive probability density functions on $\xI$:
 \begin{gather}
     d_H(f_1, f_2) \leq h e^{h /2}\\
     KL(f_1, f_2) \leq C_{KL} h^2 e^{h}(1+h)\\
     V(f_1, f_2) \leq C_{V} h^2 e^{h}(1+h)^2\\
     d_{TV}(f_1, f_2) \leq C_{TV} h e^{h/2}(1+h)
 \end{gather}
 where $h := \Vert \log(f_1)  - \log(f_2) \Vert_\infty$.
\end{lemma}

\begin{proof}
    The first three inequalities are tight, and are stated and proved in Lemma 3.1 of \cite{van_der_vaart_rates_2008}.
    The bound for $d_{TV}$ derives from the one for $KL$ via Pinsker's inequality, which is tight up to a multiplicative constant and states that $d_{TV}(f_1, f_2) \leq \sqrt{KL(f_1, f_2)}$. 
\end{proof}

As is standard in spatial statistics, we will derive the regularity of the SLGP $Y$ from the regularity of the kernel and mean of the GP $Z=\proc{Z}{\xX, t}{(\xX, t) \in D \times \xI}$ that induces $Y$. More precisely, we will demonstrate that $Y$ inherits its regularity from the canonical distance associated with the kernel.

Note that while we currently focus on the case $\dimI=1$ and $t$ is a scalar for notational simplicity, all subsequent developments and proofs apply to $\dimI>1$ by simply replacing instances of $\vert \cdot \vert$ with its multidimensional counterpart $\Vert \cdot \Vert_\infty$.

\begin{definition}[Canonical distance]
\label{def:canonicalsemimetric}
    Let $k$ be a kernel on a generic space $S$. The canonical distance associated to $k$, denoted by $d_k$ is a pseudo-metric defined by:\[ d^2_k(s, s') := k(s, s) + k(s', s') - 2 k(s, s'),  \ (s, s')\in S^2 \]
\end{definition}

Regularity conditions on kernels $k$ are generally formulates through that of the associated canonical distance. Here, we consider the following regularity condition:

\begin{condition}[Condition on kernels on  $D \times \xI$]
	\label{con:suff_k}
There exist $C, \alpha_1, \alpha_2 >0$ such that for all $\xX, \xX' \in D, t, t' \in \xI$:
	\begin{equation}
		\label{eq:Holder_k}
		d^2_{k}([\xX, t], [\xX', t']) \leq C \cdot \max(\Vert \xX - \xX' \Vert_\infty ^{\alpha_1}, \vert t - t' \vert ^{\alpha_2})
	\end{equation}
\end{condition}

\begin{remark}
\label{rk:holder}
In our setting, $D$ and $\xI$ being compact, if a covariance kernel $k$ on $D \times \xI$ satisfies Condition~\ref{con:suff_k}, it is also true that there exists $C'$ such that for all $(\mathbf{y}, \mathbf{y}') \in (D \times \xI)^2$:
\begin{equation}
\label{eq:Holder_k_sum}
d^2_{k}(\mathbf{y}, \mathbf{y}') \leq C' \cdot \Vert \mathbf{y} - \mathbf{y}' \Vert_\infty ^{\min(\alpha_1, \alpha_2)}
\end{equation}

Hence, Condition~\ref{con:suff_k} can be referred to as Hölder-type condition. Although Equation~\ref{eq:Holder_k_sum} would allow for deriving most results in the coming subsection, when deriving rates in Section \ref{sec:cty:subsec:exp} it is interesting to distinguish the regularity over $D$ from that over $\xI$ as there is a strong asymmetry between both spaces.
\end{remark}

From here, we will conduct our analysis in the setting considered in Proposition \ref{prop:underlying_mean_ker_SLGP}.

With this in mind, we show that Condition \ref{con:suff_k} is sufficient for the almost sure continuity (in sup norm) of the SLGP (Section \ref{sec:cty:subsec:AS}) as well as mean Hölder continuity of the SLGP (Section \ref{sec:cty:subsec:exp})

\subsection{Almost sure continuity of the Spatial Logistic Gaussian Process}
\label{sec:cty:subsec:AS}

First, let us remark that if a covariance kernel $k$ on $D\times \xI$ satisfies Condition~\ref{con:suff_k}, then the associated centred GP admits a version that is almost surely continuous and therefore almost surely bounded. This result, re-established in Proposition~A 
of the supplementary material \citep{gautier_supplementary_2023} for the sake of completeness, constitutes a classical result in stochastic processes literature. It is essential as it ensures the objects we will work with are well-defined. It then allows us to derive a bound for the expected value of the sup-norm of our increment field, and to leverage it for our main contribution for this section.
\begin{proposition}
	\label{prop:Sup_norm_dif_holder}
	If  a covariance kernel $k$ on $D\times \xI$ satisfies Condition~\ref{con:suff_k}, then for any $0< \delta < \dfrac{\alpha_1}{2}$, there exists a constant $K_\delta$ such that for $Z\sim \mathcal{GP}(0, k)$:
	\begin{equation}
		M(\xX, \xX') := \mathbb{E} \left[ \Vert \proc{Z}{\xX, \cdot}{} - \proc{Z}{\xX', \cdot}{} \Vert_\infty \right]  \leq K_\delta \Vert \xX - \xX' \Vert^{\alpha_1 /2 -\delta}_\infty \ \forall (\xX, \xX') \in D^2
	\end{equation}
\end{proposition}
Despite its reliance on standard results (namely Dudley's theorem), the full proof of Proposition~\ref{prop:Sup_norm_dif_holder} requires precision, to ensure that the provided bounds are sharp. As such, we decided to give the main idea here, but to refer the reader to proofs in the Appendix~\ref{app:fullproofs} for full derivations.

\begin{proof}[Main elements for proving Proposition~\ref{prop:Sup_norm_dif_holder}]
	\label{sketchproof:Sup_norm_dif_holder}
	For any $(\xX, \xX') \in D^2$, the process $\left(\proc{Z}{\xX, t}{} - \proc{Z}{\xX', t}{} \right)_{t\in \xI}$, is a Gaussian Process whose covariance kernel can be expressed as linear combination of $k$ (it is a particular case of Eq.~\ref{eq:k_can_and_k}). As such, the canonical distance associated to it (here denoted $\dxx^2$) inherits its regularity from Condition~\ref{con:suff_k}, which yields:
		\begin{gather}
		\label{eq:Holder_d_kxx}
			 \dxx^2(t, t') \leq 3C \Vert \xX - \xX' \Vert^{\alpha_1}_\infty  \quad \forall (t, t') \in \xI^2\\
			 \dxx^2(t, t') \leq 4C \vert t - t' \vert^{\alpha_2} \quad \forall (t, t') \in \xI^2
		\end{gather}
	combining these bounds with Dudley's theorem and careful numerical development yields the required result.
 \end{proof}

This bound on the expected value of the increments of a GP allows us to make a much stronger statement than the one in  Proposition~A 
of the supplementary material \citep{gautier_supplementary_2023}.

\begin{corollary}
	\label{cor:ASholder}
	If $k$ satisfies Condition \ref{con:suff_k} and  $Z \sim \mathcal{GP}(0, k)$ is measurable and separable, then the process $\left(\proc{Z}{\xX, t}{} - \proc{Z}{\xX', t}{} \right)_{t\in \xI}$ is almost surely $\beta$-Hölder continuous for any $\beta < \frac{\alpha_1}{2}$
\end{corollary}

\begin{proof}
Let $\mathfrak{B}$ be the Banach space of bounded real-valued functions on $D\times \xI$ equipped with the sup-norm.
Moreover, if $Z$ is $\mathfrak{B}$-valued, we just need to combine the bound provided by Proposition \ref{prop:Sup_norm_dif_holder} with Proposition~A 
of the supplementary material \citep{gautier_supplementary_2023}. This induces the existence of a version $\Tilde{Z}$ almost surely $\beta$-Hölder continuous. Then, $D\times \xI$ being compact it follows that $Z$ and $\Tilde{Z}$ are indistinguishable.
    
    If $Z$ is almost surely, but not surely in $\mathfrak{B}$, 
    it is easy to create $Z'$ which is indistinguishable from $Z$ and $\mathfrak{B}$-valued. Applying the reasoning above to $Z'$ yields that $Z'$ (and therefore $Z$) is almost surely $\beta$-Hölder continuous.
\end{proof}

 From thereon, we will always work with assumptions ensuring the a.s. continuity of the GPs we work with. We will consider that our GPs of interest are also $\mathfrak{B}$-valued. Indeed, 
 given an a.s. continuous GP $Z$ it is always possible to construct a \emph{surely} continuous GP $\Tilde{Z}$ (and therefore $\mathfrak{B}$-valued) indistinguishable from it.
 
\begin{theorem}
\label{th:as_SLGP}
Let us consider a centred GP $Z=\proc{Z}{\xX, t}{(\xX, t) \in D \times \xI}$ whose covariance kernel $k$ satisfies Condition~\ref{con:suff_k}.
The SLGP induced by $Z$, denoted by $Y=\proc{Y}{\xX, t}{(\xX, t) \in D \times \xI}$ is in $\indpospdf(\xI; D)$ and is almost surely $\beta$-Hölder continuous with respect to $\Vert \cdot \Vert_\infty$, for any $\beta < \frac{\alpha_1}{2}$. 
\end{theorem}

\begin{proof}
	First, note that under Condition~\ref{con:suff_k}, $Z$ is almost surely continuous (and hence a.s. bounded). It follows from it that $Y$ is almost surely in $\indpospdf(\xI; D)$. 
Now observe that for any $t \in \xI$, we have:
\begin{equation}
		\proc{Z}{\xX, t}{} - \Vert \proc{Z}{\xX, \cdot}{} -  \proc{Z}{\xX', \cdot}{}\Vert_\infty \leq \proc{Z}{\xX', t}{} \leq \proc{Z}{\xX, t}{} +\Vert \proc{Z}{\xX, \cdot}{} -  \proc{Z}{\xX', \cdot}{}\Vert_\infty
\end{equation}
 with $\Vert \proc{Z}{\xX, \cdot}{} -  \proc{Z}{\xX', \cdot}{}\Vert_\infty$ being possibly infinite on the $P$-null-set where $Z$ is not continuous. As such, for any $t \in \xI$:
	\begin{equation}
		\left\vert \dfrac{e^{\proc{Z}{\xX', t}{}}}{\int_\xI e^{ \proc{Z}{\xX', u}{} } \,d\lambda(u) } - 
		\dfrac{e^{\proc{Z}{\xX, t}{}}}{\int_\xI e^{\proc{Z}{\xX, u}{}} \,d\lambda(u)} \right\vert \leq \dfrac{e^{\proc{Z}{\xX, t}{}}}{\int_\xI e^{\proc{Z}{\xX, u}{}} \,d\lambda(u)} \left[ e^{ 2 \Vert \proc{Z}{\xX, \cdot}{} -  \proc{Z}{\xX', \cdot}{}\Vert_\infty} -1 \right]
	\end{equation}
	
	By convexity of the exponential, we find that:
	
	\begin{equation}
		\left\Vert \proc{Y}{\xX, \cdot}{} - \proc{Y}{\xX', \cdot}{} \right\Vert_\infty 
		 \leq \dfrac{e^{2 \Vert Z \Vert_\infty}}{\lambda(\xI)} \left[ 2 \Vert \proc{Z}{\xX, \cdot}{} -  \proc{Z}{\xX', \cdot}{}\Vert_\infty + O(\Vert \proc{Z}{\xX, \cdot}{} -  \proc{Z}{\xX', \cdot}{}\Vert_\infty^2) \right]
	\end{equation}
	Combining the almost-sure boundedness of $Z$ with Corollary~\ref{cor:ASholder} and the domain's compactness concludes the proof.
\end{proof}

\begin{remark}
For notational simplicity, we focused on the case where $Z$ is a centred GP, but all properties are easily extended to the case where the mean of $Z$ is $\beta$-Hölder continuous for any $\beta < \frac{\alpha_1}{2}$. 
\end{remark}

From Proposition~\ref{prop:Sup_norm_dif_holder}, we are also able to derive an analogue to scalar's mean square continuity, presented in the following section.

\subsection{Mean power continuity of the Spatial Logistic Gaussian Process}
\label{sec:cty:subsec:exp}
 Indeed, we also leverage the bound on the expected value of the sup-norm of our increment field to show that the Hölder conditions on $k$ and $\kincrement$ are sufficient conditions for the mean power continuity of the SLGP.

\begin{theorem}[Sufficient condition for mean power continuity of the SLGP]
	\label{th:Expected_quadratic_cty}
	Consider the SLGP $Y=\proc{Y}{\xX, t}{(\xX, t) \in D \times \xI}$ induced by a centred GP $Z=\proc{Z}{\xX, t}{(\xX, t) \in D \times \xI}$ with covariance kernel $k$. Assume that $k$ satisfies Condition~\ref{con:suff_k}.
	
	Then, for all $\gamma>0$ and $0<\delta < \gamma\alpha_1/2$ (for Equations~\ref{eq:cty_1}-\ref{eq:cty_0}, resp. $0<\delta < \gamma\alpha_1$ for Equations~\ref{eq:cty_2}-\ref{eq:cty_3}), there exists $K_{\gamma, \delta}>0$ such that for all $\xX, \xX' \in D^2$:
	\begin{align}
            \mathbb{E} \left[ d_{H}(\proc{Y}{\xX, \cdot}{}, \proc{Y}{\xX', \cdot}{})^\gamma \right] \leq K_{\gamma, \delta}  \Vert \xX - \xX' \Vert^{\gamma \alpha_1 /2 -\delta}_\infty \label{eq:cty_1}\\
            \mathbb{E} \left[ V(\proc{Y}{\xX, \cdot}{}, \proc{Y}{\xX', \cdot}{})^\gamma \right] \leq K_{\gamma, \delta}  \Vert \xX - \xX' \Vert^{\gamma \alpha_1/2 -\delta}_\infty \label{eq:cty_0}\\
            \mathbb{E} \left[ KL(\proc{Y}{\xX, \cdot}{}, \proc{Y}{\xX', \cdot}{})^\gamma \right] \leq K_{\gamma, \delta}  \Vert \xX - \xX' \Vert^{\gamma \alpha_1 -\delta}_\infty \label{eq:cty_2}   \\
            \mathbb{E} \left[ d_{TV}(\proc{Y}{\xX, \cdot}{}, \proc{Y}{\xX', \cdot}{})^\gamma \right] \leq K_{\gamma, \delta}  \Vert \xX - \xX' \Vert^{\gamma \alpha_1 -\delta}_\infty \label{eq:cty_3} 
        \end{align}
\end{theorem}

 The main addition of this theorem, compared to the Proposition \ref{th:as_SLGP} is that it provides some control on the modulus of continuity. In Theorem~\ref{th:Expected_quadratic_cty}, we add to the almost-sure Hölder continuity by also providing rates on the dissimilarity between SLGPs considered at different $\xX$'s. We give here the sketch of proof and refer the reader to appendix~\ref{proof:Expected_quadratic_cty} for detailed derivations.

\begin{proof}[Main elements for proving Theorem \ref{th:Expected_quadratic_cty}]
The core idea of this proof is to leverage Lemma~\ref{pty:bound_logtransf_sup} and to apply Fernique's theorem. Careful analysis and further derivations enable us to prove that we have the following (tight) upper-bounds:
\begin{equation}
\label{eq:boundsSLGP-GP}
		\begin{array}{c}
			\mathbb{E} \left[ 
			d_{H}(\proc{Y}{\xX, \cdot}{}, \proc{Y}{\xX', \cdot}{})^\gamma 
			\right] \leq \kappa_{\gamma} \mathbb{E} \left[ \Vert \proc{Z}{\xX, \cdot}{} - \proc{Z}{\xX', \cdot}{} \Vert_\infty \right]^{\gamma} \\
   \mathbb{E} \left[ 
			V(\proc{Y}{\xX, \cdot}{}, \proc{Y}{\xX', \cdot}{})^\gamma 
			\right] \leq \kappa_{\gamma} \mathbb{E} \left[ \Vert \proc{Z}{\xX, \cdot}{} - \proc{Z}{\xX', \cdot}{} \Vert_\infty \right]^{\gamma} \\
			\mathbb{E} \left[ 
			KL(\proc{Y}{\xX, \cdot}{}, \proc{Y}{\xX', \cdot}{})^\gamma 
			\right] \leq \kappa_{\gamma} \mathbb{E} \left[ \Vert \proc{Z}{\xX, \cdot}{} - \proc{Z}{\xX', \cdot}{} \Vert_\infty \right]^{2 \gamma} \\
			\mathbb{E} \left[ 
			d_{TV}(\proc{Y}{\xX, \cdot}{}, \proc{Y}{\xX', \cdot}{})^\gamma 
			\right] \leq \kappa_{\gamma} \mathbb{E} \left[ \Vert \proc{Z}{\xX, \cdot}{} - \proc{Z}{\xX', \cdot}{} \Vert_\infty \right]^{2 \gamma} \\
		\end{array}
	\end{equation}
We combine these inequalities with Proposition~\ref{prop:Sup_norm_dif_holder} to conclude the proof.
\end{proof}

\begin{remark}
\label{rmrk:uniform_approximation_GP}
	The proof of this theorem consists in getting to the inequalities in Eq.~\ref{eq:boundsSLGP-GP} and then leveraging Proposition~\ref{prop:Sup_norm_dif_holder}. It is noteworthy that the exact same proof structure can be applied, for instance, to prove that for a SLGP $Y = \slogt [Z]$ and a density field $f$ obtained by spatial logistic density transformation of a function $g$, $f =  \slogt [g]$, if $\mathbb{E} \left[ \Vert \proc{Z}{\xX, \cdot}{} - g(\xX, \cdot) \Vert_\infty^{\gamma} \right]  \rightarrow 0$ then for all $\xX$:
	
	\begin{equation}
		\begin{array}{c}
			\mathbb{E} \left[ 
			d_{H}(f(\xX, \cdot), \proc{Y}{\xX, \cdot}{})^\gamma 
			\right] \rightarrow 0 \\
			\mathbb{E} \left[ 
			V(f(\xX, \cdot), \proc{Y}{\xX, \cdot}{})^\gamma 
			\right] \rightarrow 0 \\
			\mathbb{E} \left[ 
			KL(f(\xX, \cdot), \proc{Y}{\xX, \cdot}{})^\gamma 
			\right] \rightarrow 0\\
			\mathbb{E} \left[ 
			d_{TV}(f(\xX, \cdot), \proc{Y}{\xX, \cdot}{})^\gamma 
			\right] \rightarrow 0\\
		\end{array}
	\end{equation}
	Hence making these bounds applicable in the context of uniform approximation by a (SL)GP. 
\end{remark}

\section{Applications in density field estimation}

\label{sec:app}

In addition to studying the mathematical properties of the class of models at hand, we propose an implementation for density field estimation which we summarize in section~\ref{sec:app:subsec:implementation}.

In section~\ref{sec:app:subsec:an:subsubsec:reg}, we compare empirical versus theoretical rates provided by Theorem~\ref{th:Expected_quadratic_cty}. To do so, we use unconditional realisations of SLGPs obtained by transforming centered GPs with kernels of various Hölder exponents. Then, in section~\ref{sec:app:subsec:an:subsubsec:impact}, we illustrate the SLGPs potential in estimating the true data generative process on analytical test-cases. Finally, we showcase the potential of our model by applying it to a real-world dataset of temperatures in Switzerland in section~\ref{sec:app:subsec:meteo}.

\subsection{A brief overview of our implementation choices}
\label{sec:app:subsec:implementation}

This section only presents a brief overview of the choices we made in implementing the model. We refer the reader to \cite{gautier_modelling_2023} for more details, as well as supporting elements in section~5 

In all that follows, we consider that our dataset consits in $n$ couples of locations and observations $\{(\xX_i, t_i)\}_{1 \leq i \leq n}\in ( D\times \xI )^n$. We assume the $t_i$'s are realisations of some independent random variables $T_i$, and denote by $f_0(\xX_i, \cdot)$ the (unknown) density of $T_i$. We will denote $\mathbf{T}_n := (T_i)_{1 \leq i \leq n}$ and $\mathbf{t}_n := (t_i)_{1 \leq i \leq n}$ the corresponding vectors.

For a suitable kernel $k$ (for instance $k$ satisfying Condition~\ref{con:suff_k}), we consider the Gaussian Process $Z=\proc{Z}{\xX, t}{(\xX, t) \in D\times \xI} \sim \mathcal{GP}(0, k)$ assumed to be exponentially integrable. Let us denote $Y=\proc{Y}{\xX, t}{(\xX, t) \in D\times \xI}$ the SLGP induced by $Z$. Assuming that the observations are drawn from $Y$ yields the likelihood:

\begin{equation}
    \mathcal{L}( \mathbf{t}_n \vert Y) = \prod\limits_{i=1}^n \proc{Y}{\xX_i, t_i}{}
\end{equation}
It is also possible to work with $Z$ directly:
\begin{equation}
	\label{eq:like_data}
	\mathcal{L}( \mathbf{t}_n \vert Z) =  \prod\limits_{i=1}^n \dfrac{e^{\proc{Z}{\xX_i, t_i}{}}}{\int_\xI e^{\proc{Z}{\xX_i, u}{}} \,du } 
\end{equation}
This latter formulation is easier to work with than the prime, but implementation of this density field estimation still causes two main issues. The first one being that the integral in Equation~\ref{eq:like_data} involves values of $Z$ over the whole domain. This infinite dimensional object makes likelihood-based computations cumbersome. 
The other issue lies on the fact that in all realistic cases, GP depends on some unknown hyperparameters that need to be estimated.

\paragraph{Handling the dimensionality of the problem:} 
A simple approach to this problem would be to consider a SLGP discretized both in time and space, by first specifying a fine grid over which we want to work. However, this approach does not scale well in dimension and requires knowing where one wants to predict before performing the estimation. Instead, we propose a parametrisation in frequency, that consists in considering finite rank Gaussian Processes, i.e. processes that can be written as:
\begin{equation}
	\label{eq:GPextension}
	\proc{Z}{\xX, t}{} = \sigma \sum_{j=1}^p f_j(\xX, t) \varepsilon_j , \ \forall \xX \in D, \ t\in \xI
\end{equation}
where $p \in \mathbb N$, the $f_i$ are functions in $L^2(D\times \xI)$, $\sigma >0$ is the standard deviation of the process, and the $\varepsilon_i$'s are i.i.d. $\mathcal{N}(0, 1)$. 

Also note that we use the notation $\varepsilon$ and variations thereof for the random variables, and $\epsilon$ for instantiated values.

Under this parametrisation, we have a finite-dimensional posterior:
\begin{equation}
	\label{eq:int_post_finite}
	\pi( \boldsymbol \epsilon \vert \mathbf{T}_n = \mathbf{t}_n) \propto \phi( \boldsymbol \epsilon) \prod\limits_{i=1}^n
	\dfrac{e^{\sigma \sum_{j=1}^p f_j(\xX_i, t_i) \varepsilon_j}}{\int_\xI e^{\sigma \sum_{j=1}^p f_j(\xX_i, u) \epsilon_j} \,du } 
\end{equation}
Where $\phi$ denotes the multivariate density of the standard normal distribution.

We implement the Maximum A Posterior (MAP) estimation, as well as MCMC estimation. The later delivers a probabilistic prediction of the considered density fields, and allows us to approximately sample from the posterior distribution on $Z$. This generative model can be leveraged to quantify uncertainty on the obtained predictions.

\paragraph{Handling the hyper-parameters:} Assuming that the GP is parametrised by a variance parameter $\sigma^2$ as well as $\dimD+1$ length-scale parameters $\rho_t, \rho_1, ..., \rho_{\dimD}$ through: 
\begin{equation}
	\label{eq:GPextensionPara}
	\proc{Z}{\xX, t}{} = \sigma \sum_{j=1}^p f_j(x_1 / \rho_1, ..., x_{\dimD} / \rho_{\dimD}, t/\rho_t) \varepsilon_j , \ \forall \xX \in D, \ t\in \xI
\end{equation}
we use an heuristic for $\sigma^2$ and a grid search for the $\rho$'s.

\textbf{Variance heuristic: }This part consists in selecting (through numerical simulations) a variance that ensures numerical stability of the SLGP. In our case, we typically select $\sigma^2$ such that $\mathbb{E}[\max_{\xX \in D} \vert \max_{t \in \xI} \proc{Z}{\xX, t}{} - \min_{t \in \xI} \proc{Z}{\xX, t}{}\vert ] \leq 5$. This heuristic controls the range of values that the SLGP can take, typically restricting it to $[0, e^5\approx 148]$. This step is not performed conditionally on any data, and is here to ensure that we specified a model that does not present numeric overflow. Motivating examples underlying this heuristic are available in a vignette \citep{gautier_vignette_2024_varheuri}.

\textbf{Length-scale selection: }As for the length-scales, we use a Bayesian approach where we specify a prior $\pi$ over $\boldsymbol \rho$, and maximise the joint posterior:

\begin{equation}
	\label{eq:int_post_joint_finite}
 \hspace*{-8pt} \pi( \boldsymbol \epsilon; \boldsymbol \rho \vert \mathbf{T}_n = \mathbf{t}_n) \propto \pi(\boldsymbol \rho) \phi( \boldsymbol \epsilon) \prod\limits_{i=1}^n
	\dfrac{e^{\sigma \sum_{j=1}^p f_j(x_{1, i} / \rho_1, ..., x_{\dimD, i} / \rho_{\dimD}, t_i/\rho_t) \epsilon_j }}{\int_\xI e^{\sigma \sum_{j=1}^p f_j(x_{1, i} / \rho_1, ..., x_{\dimD, i} / \rho_{\dimD}, u/\rho_t) \epsilon_j} \,du } 
\end{equation}
Then, we use the length-scales realising the MAP in any subsequent MCMC estimation.

\paragraph{Summary of the implementation}

In our implementation, we used finite rank Gaussian Processes defined through the Random Fourier Features (RFF)  approach, where the core idea is to sample frequencies that allow to approximate a kernel and use a basis functions made of sines and cosines with these frequencies \citep{rahimi_random_2008, rahimi_weighted_2009}. 
We also focus our implementation on Matérn-type kernels, a family of kernels that are well-suited for RFF due to their broad expressiveness and the closed form formula for their spectral density. 
For further detail on these topics, we refer the reader to the supplementary material and references therein: section~5.2
is a primer on RFF, and section~5.3
gives more details on the Matérn kernels.

\emph{Main steps of our implementation:}
\begin{itemize}
	\item \textbf{Model specification:} Select a stationary kernel from the Matérn family, a  number $p \in \mathbb{N}$ of frequencies to sample, and a prior distribution for the length-scale. In light of results from \cite{van_der_vaart_adaptive_2009}, we opt for an inverse Gamma prior.
 
        \item \textbf{Defining the basis function:} Draw $p$ i.i.d. samples from $k$'s spectral density $s(\cdot)$, denoted $\boldsymbol \xi_i \in \xR^{\dimD + 1}$. For $1\leq i\leq p$, define $f_i(\xX, t) = \cos(2\pi \boldsymbol \xi_i^\top[\xX, t])$, and $f_{i+p}(\xX, t) = \sin(2\pi \boldsymbol \xi_{i-p/2}^\top [\xX, t])$.
        
        \item \textbf{Determine the variance} Use the heuristic rule above to determine a value $\sigma^2_{stable}$ of the variance that ensures numerical stability. This step is not data-dependant.
        
        \item \textbf{Perform MAP estimation while estimating the length-scales} Find the $\boldsymbol \epsilon^*$ and $\boldsymbol \rho^*$ that maximise the posterior in Equation~\ref{eq:int_post_joint_finite}.
        If only a point estimate of the density field was required, the estimation is done and the predictor associated is:
        \begin{equation}
            \proc{\hat Y}{\xX, t}{}^{\text{MAP}} = \dfrac{e^{\sigma_{stable} \sum_{j=1}^p f_j(x_{1} / \rho_1^*, ..., x_{\dimD} / \rho_{\dimD}^*, t/\rho_t^*) \epsilon_j^* }}{\int_\xI e^{\sigma_{stable} \sum_{j=1}^p f_j(x_{1} / \rho_1^*, ..., x_{\dimD} / \rho_{\dimD}^*, u/\rho_t^*) \epsilon_j^*} \,du }
        \end{equation}
        \item \textbf{(If necessary) Perform MCMC:}
	For the MCMC, due to heavier computational load we recommend keeping the hyper-parameters set to their previously determined values $\sigma_{stable}^2$ and $\boldsymbol \rho^*$, and to only perform the inference over $\boldsymbol \varepsilon$. Our current implementation uses the No-U-Turn Sampler algorithm implemented in the RStan package \citep{stan_development_team_rstan_2024}. This MCMC step returns $N$ draws $(\boldsymbol \epsilon_i)_{1\leq i \leq N}$ from the posterior of $\boldsymbol \varepsilon$, which in turns yields a probabilistic prediction of the density field, since we have $N$ draws of the SLGP:
 \begin{equation}
            \proc{\hat Y}{\xX, t}{}^{\text{MCMC}, i} = \dfrac{e^{\sigma_{stable} \sum_{j=1}^p f_j(x_{1} / \rho_1^*, ..., x_{\dimD} / \rho_{\dimD}^*, t/\rho_t^*) \epsilon_{i, j} }}{\int_\xI e^{\sigma_{stable} \sum_{j=1}^p f_j(x_{1} / \rho_1^*, ..., x_{\dimD} / \rho_{\dimD}^*, u/\rho_t^*) \epsilon_{i, j}} \,du }
        \end{equation}	
\end{itemize}
This implementation allows us to perform density field estimation, as we are about to demonstrate in the coming section.

\subsection{Synthetic test cases: how regularity influences learning.}
\label{sec:app:subsec:an}

\subsubsection{Some kernels and the resulting spatial regularity}
\label{sec:app:subsec:an:subsubsec:reg}

We consider two commonly used covariance kernels, and visualize how their continuity modulus affects their expected power continuity. We focus on the setting where $D=\xI=[0, 1]$. For $y=[x, t], y'=[x', t'] \in \xR^2$ we consider:
\begin{itemize}
    \item The exponential kernel: $\displaystyle k( y, y') :=e^{-\Vert y-y'\Vert_2}$, Hölder exponent $ \alpha_{1}= 1$.
    \item The Gaussian kernel: $k(y, y') :=e^{-\Vert y-y'\Vert_2^{2} /2}$, Hölder exponent $\alpha_{1} = 2$.
\end{itemize}
Other commonly used kernels such as Matérn 3/2 or 5/2 are not displayed here since they share the Hölder exponent of the Gaussian kernel.
	
By drawing 1000 unconditional realisations of SLGPs induced by centred GPs with the corresponding kernels, we can represent a re-scaled version of $\mathbb{E}\left[ \mathfrak{D} \left( \SLGPM{0}{}, \SLGPM{\xX'}{}\right)^\gamma \right]$ for the four dissimilarities $\mathfrak{D}$ considered in this paper and varying $\gamma$. We also represent the corresponding theoretical rate. Re-scaling is used solely to allow all curves to appear on the same plots.

\begin{figure}[H]
    \centering
    \includegraphics[width=0.9\linewidth]{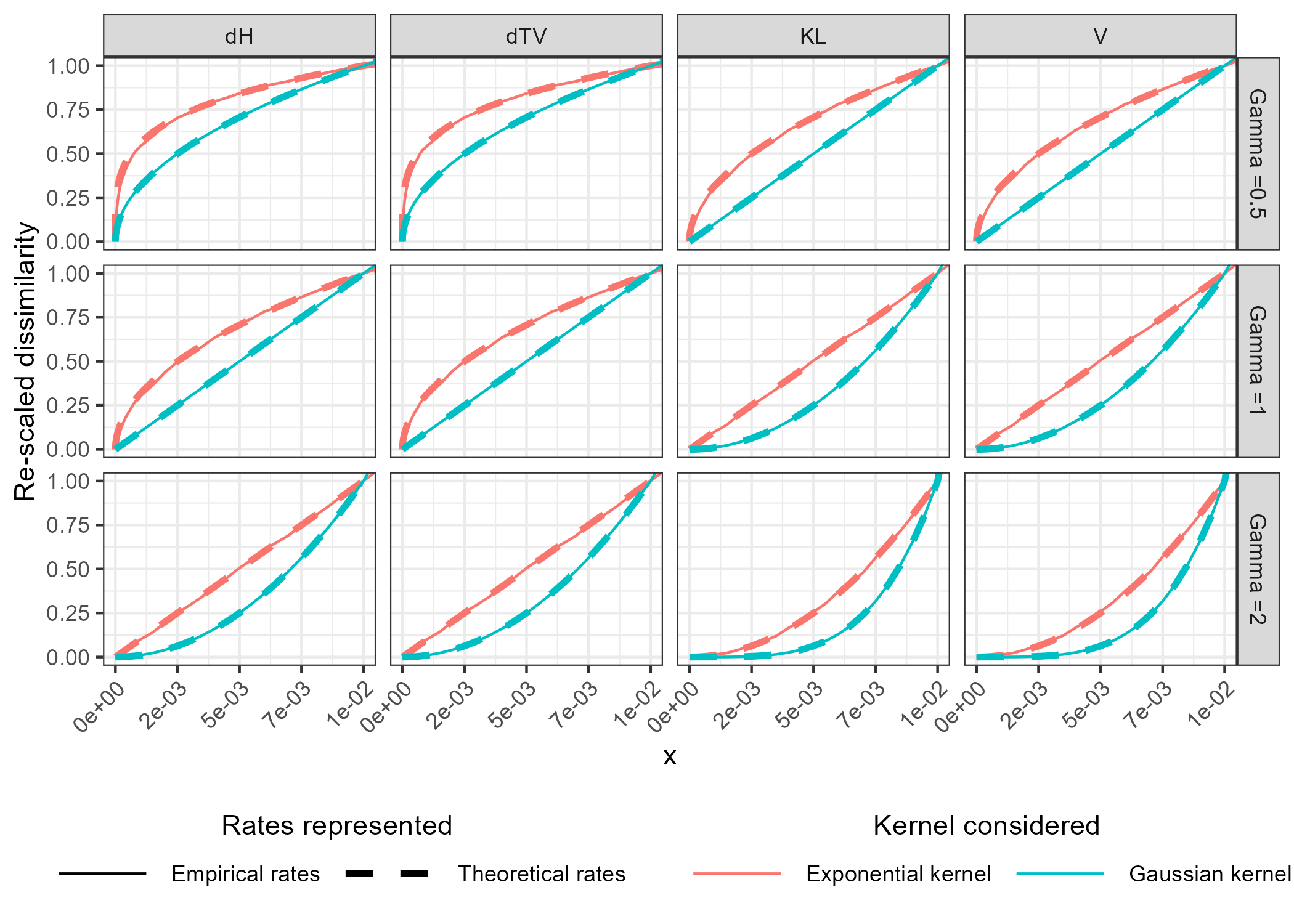}
    \caption{Visualising $\mathbb{E}\left[ \mathfrak{D} \left( \SLGPM{0}{}, \SLGPM{\xX'}{}\right)^\gamma \right]$ (plain lines) and the theoretical bound (dotted lines) for both kernels, all four dissimilarities and $\gamma \in \{0.5, 1, 2\}$.}
    \label{fig:ratesHolder}
\end{figure}

Figure~\ref{fig:ratesHolder} supports our claim that the bounds obtained previously in the paper are tight. For more details on this topic and to access the codes underlying this application, we refer the readers to the code, made available at \cite{gautier_vignette_2024_ctyrates}.

We now continue working on synthetic fields, and check that our SLGP models allow for sample-based distributional learning.

\subsubsection{Impact on learning the density field}
\label{sec:app:subsec:an:subsubsec:impact}
For $D=\xI=[0, 1]$, we consider four density valued-fields, perfectly known, represented in Figure \ref{fig:anfield}. These references fields are realisations of SLGPs with prescribed spatial regularities and known hyper-parameters.

\vspace{-8pt}
\begin{figure}[H]
	\centering
	\includegraphics[width=0.925\linewidth]{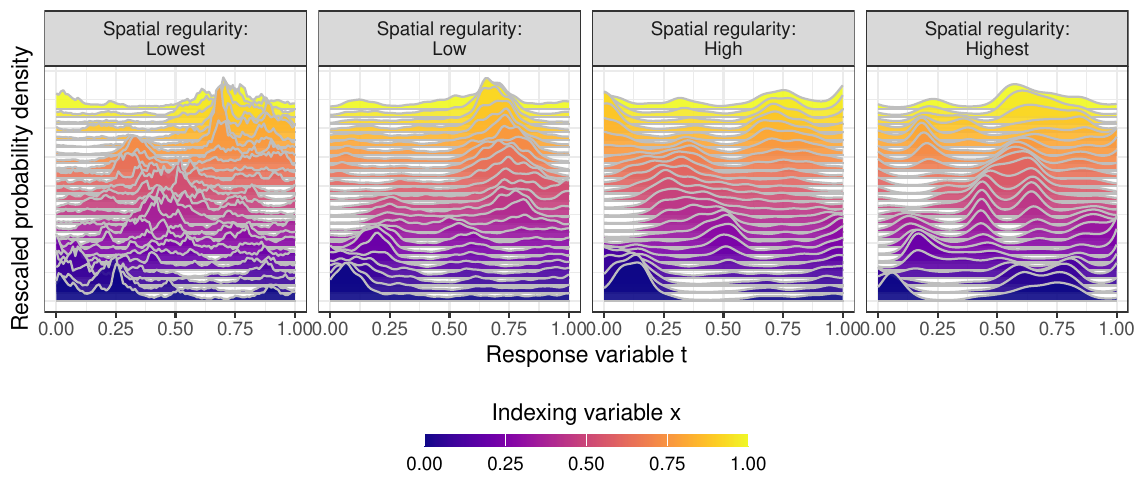}
	\caption{Representation of the four density fields used as reference: probability density functions over slices at some prescribed $\xX \in [0, 1]$.}
	\label{fig:anfield}
\end{figure}\vspace{-6pt}

We run the density field estimation without hyper-parameters estimation (as they are known). Figures displaying Maximum A Posteriori (MAP) estimators are available in Figure~\ref{fig:anfield1:res} for the first reference field and in Section~6
of the supplementary material, Figures~A, B and C
for the other three. 

\vspace{-8pt}
\begin{figure}[H]
	\centering
	\includegraphics[width=0.925\linewidth]{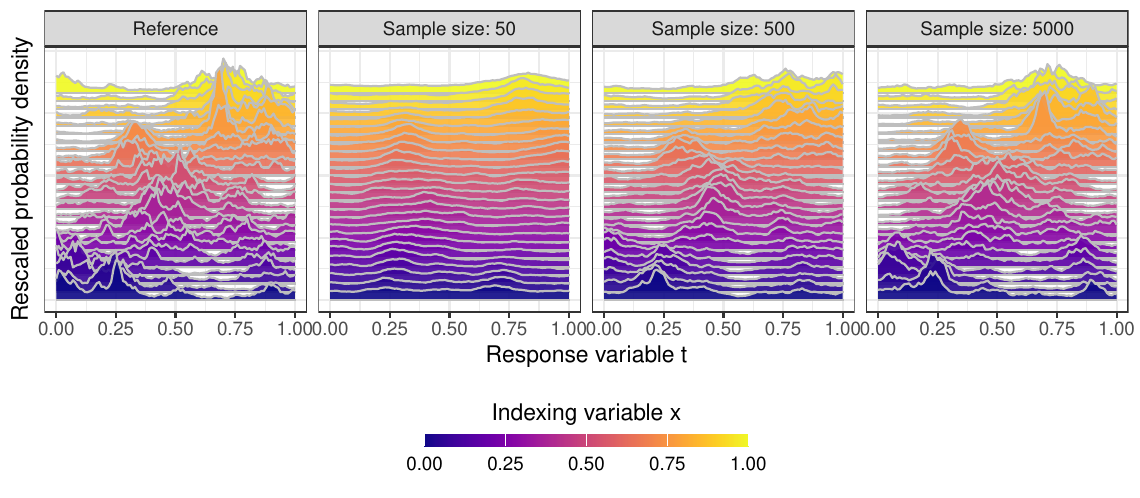}
	\caption{Results for the first reference field (lowest spatial regularity): Showing the MAP estimator for varying sample sizes, with samples scattered across space.}
	\label{fig:anfield1:res}
\end{figure}\vspace{-6pt}

We visually assess that the goodness of fit of our density estimation procedure increases with the number of available observations. Moreover, since we only consider finite rank GP, the number of basis functions used may also determine how precise our estimation can be. In order to quantify
the prediction error for different sample sizes and number of basis functions, we define an Integrated Hellinger distance to measure dissimilarity between two probability density valued fields $f(\xX, \cdot)$ and $f'(\xX, \cdot)$:

\begin{equation}
	d_{IH}^2(f(\xX, \cdot) , f'(\xX, \cdot) ) = \frac{1}{2}\int_D \int_\xI \left( \sqrt{f(\mathbf{v}, u)} - \sqrt{f'(\mathbf{v}, u))}  \right)^2 \,du \,d\mathbf{v}
\end{equation} 
In Fig. \ref{fig:isH}, we display the distribution of $d_{IH}$ between true and estimated fields for various sample sizes and SLGP orders, for SLGPs matching the respective spatial regularities of the reference fields. We see that for small sample sizes, the number of basis functions has little influence, but that it becomes limiting when more observations are available. Indeed, the SLGP models relying on the smallest numbers of basis functions appear to struggle to capture small scale variations.

\begin{figure}[H]
	\centering
	\includegraphics[width=\linewidth]{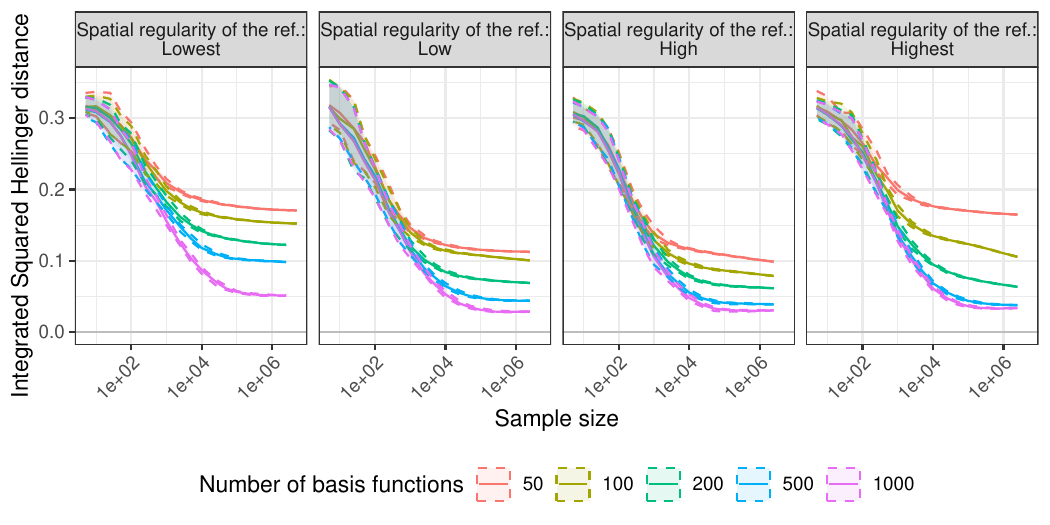}
	\caption{Integrated Hellinger distance distribution for different sample sizes and process orders. SLGP's spatial regularity match that of the corresponding reference field, and the uncertainty is estimated through 100 repetitions of each experiment.}
	\label{fig:isH}
\end{figure}

For more insights on this topic and to access the codes underlying this application, we refer the readers to \cite{gautier_vignette_2024_postconst}. Additionally, we note that the current errors converge to positive values rather than zero. Our investigations suggest that this is an artifact of the sampling strategy employed, which falls outside the scope of this work and will be addressed in future studies.

We now turn our attention to a final application, where SLGP modeling is used for statistical inference on real-world data.

\subsection{Application to a meteorological dataset}
\label{sec:app:subsec:meteo}
This section presents an application using a dataset of daily average temperatures in Switzerland for the year 2019. While not intended as a true forecasting application, this example serves to illustrate the applicability of the SLGP density estimation model on real-world data. The full code for this analysis is available on our website \citep{gautier_vignette_2024_meteo}, providing a practical reference for reproducing and extending the analysis.

\paragraph{Data and problem considered: } The dataset consists of daily temperature averages from 29 stations across Switzerland. These stations, as shown in Figure~\ref{fig:switzerlandstationswithoutbe}, are distributed across the country and capture a range of altitudes, latitudes, and longitudes. An example of the available data is provided in section~6.2
 of the supplementary material \citep{gautier_supplementary_2023}. The SLGP model is fitted to data from all but three of the stations, leaving these out for testing the generalization of the model. Since date-specific trends are not considered, this analysis focuses solely on marginal temperature distributions.

\begin{figure}[H]
	\centering
	\includegraphics[width=0.8\linewidth]{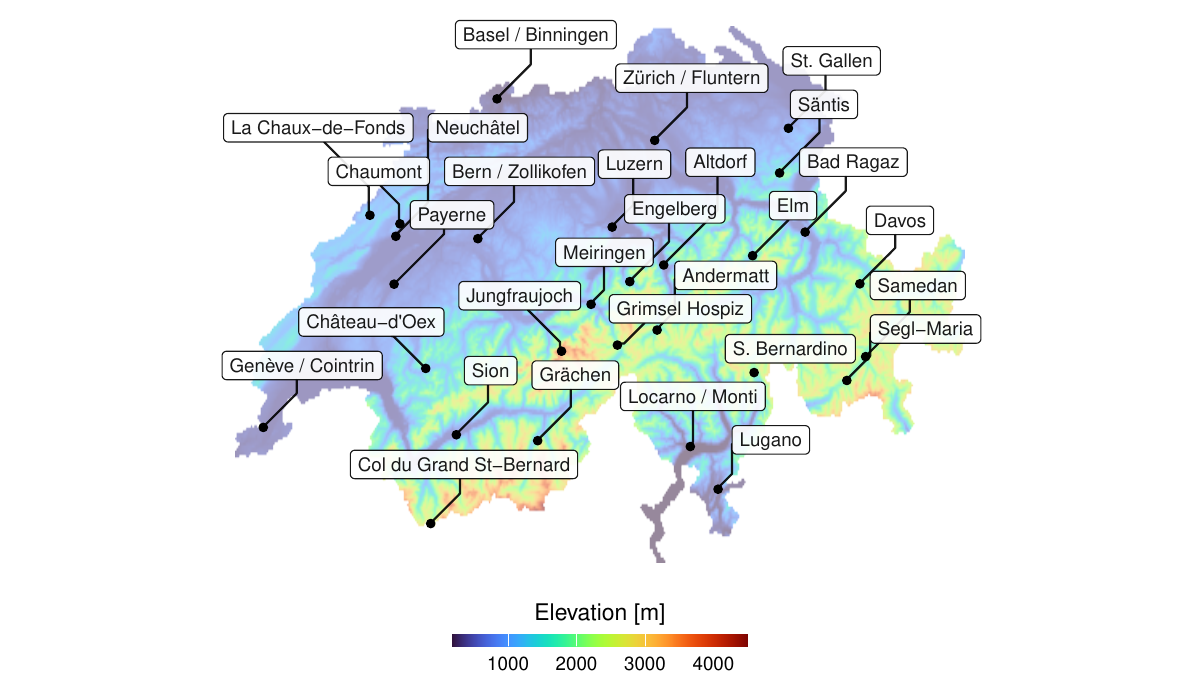}
	\caption{Map of Switzerland showing the 29 Stations present in the data-set.}
	\label{fig:switzerlandstationswithoutbe}
\end{figure}

The temperature data-set is provided by MeteoSwiss \citep{meteorology_climatological_2019} and the topographical data is provided by the Swiss Federal Office of Topography \citep{swisstopo_digital_2019}.

\paragraph{Practicalities of the SLGP modelling: } We applied a finite-rank GP approach using 250 random Fourier features (equivalent to 500 basis functions) based on a multivariate Matérn 5/2 kernel. As described in the previous section, the variance parameter was calibrated independently form any data and ensures numerical stability of the prior. The length-scales were selected through a grid search, where we tested a range of length-scale values and performed MAP estimation for each, ultimately retaining the configuration that maximized the posterior. The profiles of the negative-log-posterior values for the lengthscales can be found in section~6.2
Figure~8
. With this approach, we identify promising values of the lengthscale for latitude and longitude to be at $40\%$ of the range, while it is at $15\%$ for the altitude and $7.5\%$ for the temperature. 

\begin{figure}[H]
	\centering
	\includegraphics[width=0.9\linewidth]{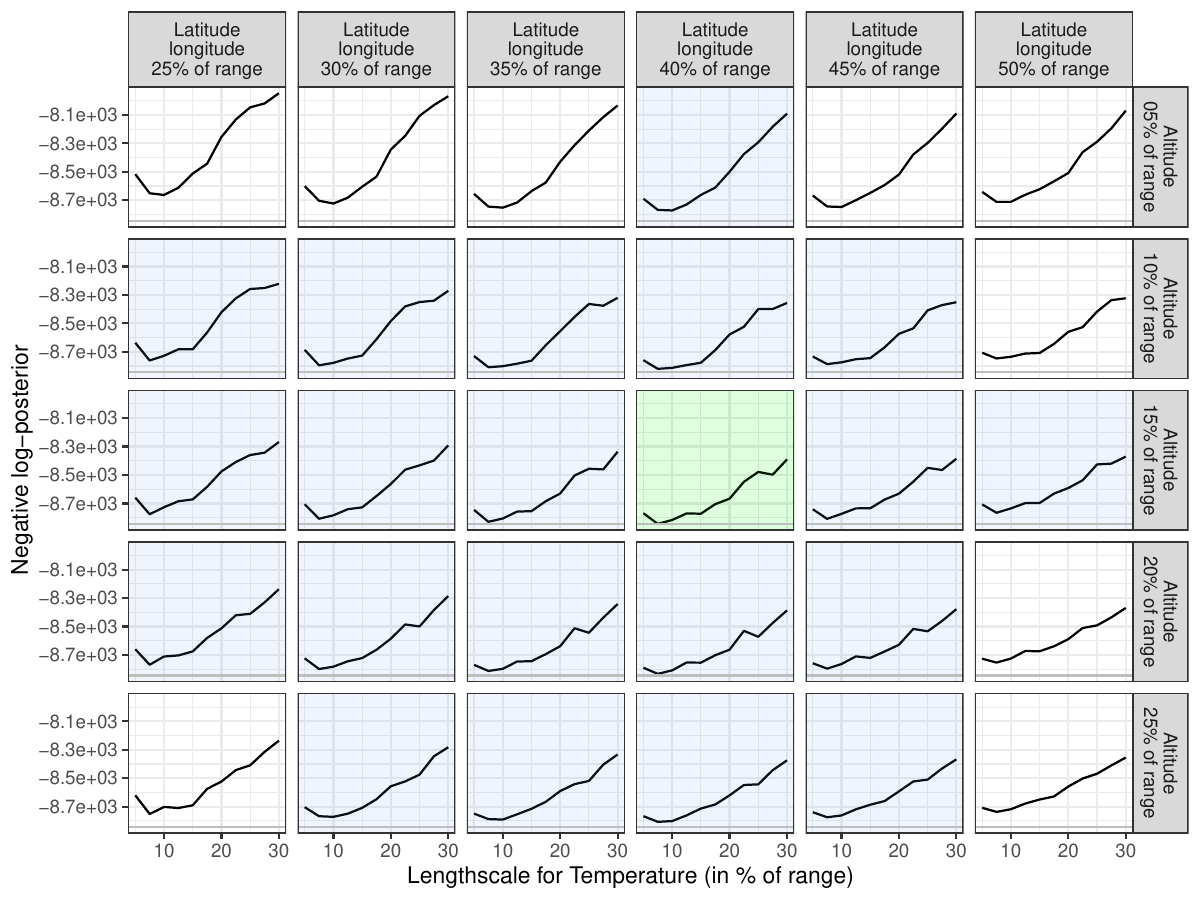}
	\caption{Showing the values of the negative log-posterior when varying the lengthscale parameters for Latitude/Longitude, Altitude and Temperature. The panel achieving the minimum is highlighted in green, the panels whose smallest values are less than 1\% away from the minimum are in light blue, emphasizing the relative flatness of the optimisation problem.}
 \label{fig:negloglikeprofilemeteo}
\end{figure}

First, we present estimation results at some specific stations to facilitate comparison with observed data. However, SLGP modeling offers more comprehensive capabilities, enabling predictions of the density field over the entire region (in this case, all of Switzerland). Results at other stations, not represented here, are available in figures~E, F, G and H
of the supplementary material. 

In  Figure~\ref{fig:meteo_at_stations_with_data}, we display the results of the SLGP estimation for 5 stations in the training dataset. For all the considered stations, the MAP estimate and the MCMC mean follow the histograms quite closely. 

It also appears that the MCMC draws have some variability, and while not all draws are able to capture precisely the locations of the modes, the overall structure is well understood and all the draws put probability mass in the relevant region of space. 

\begin{figure}[H]
	\centering
	\includegraphics[width=0.9\linewidth]{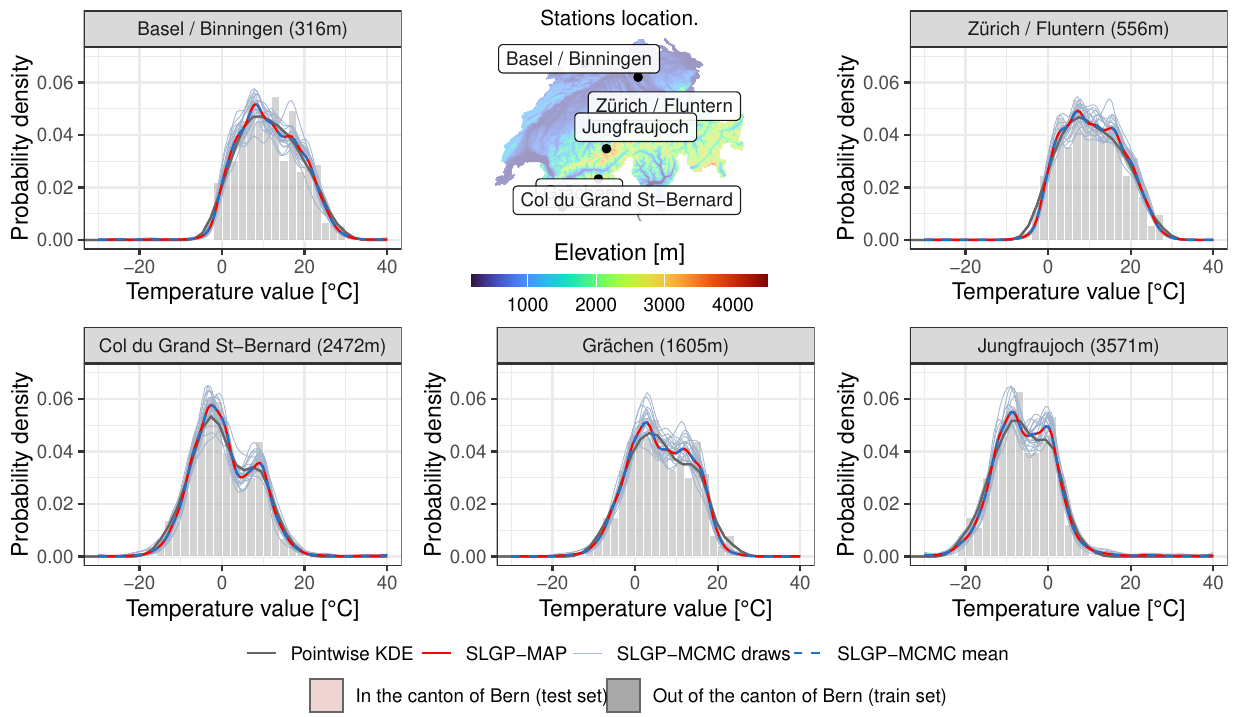}
	\caption{SLGP trained on 26 meteorological stations (365 observations each). We display for 5 stations in the training set the empirical distribution (histogram and pointwise Kernel Density Estimator) and SLGP-based estimation. }
	\label{fig:meteo_at_stations_with_data}
\end{figure}

This is expected: 365 replicates per station yield a point-wise estimator that reasonably captures these structures. Indeed, the pointwise Kernel Density Estimators yield visually similar results. The main difference being that KDE tends to smooth the empirical distributions, while the SLGP does not as much (e.g. Basel / Binningen or Col du Grand St-Bernard for such events).



This motivates us to look at predictions for the three stations that we left out of the data set, to assess the extrapolation capabilities of SLGP models. 
We observe that the resulting estimates, displayed in Figure~\ref{fig:meteo_at_stations_without_data}, present more variability at stations left out of the training set, a desirable feature.  

\vspace{-8pt}
\begin{figure}[H]
	\centering
	\includegraphics[width=0.9\linewidth]{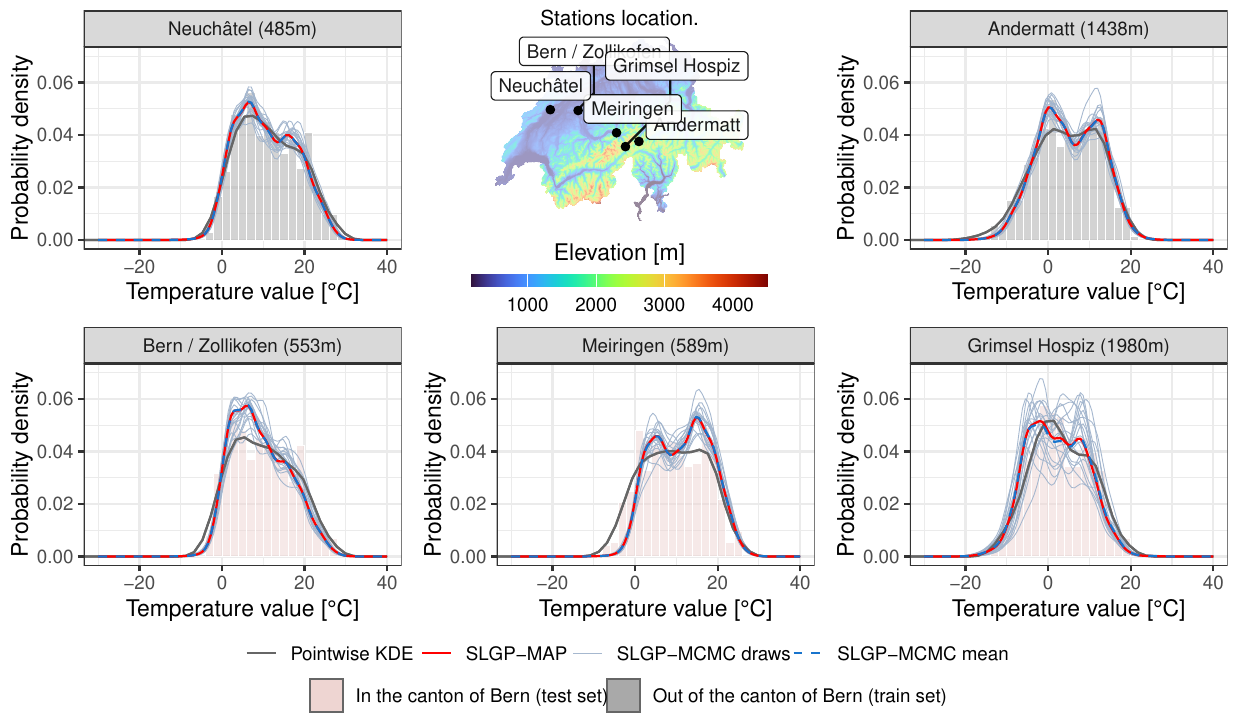}
	\caption{We display for 3 stations in the test set and 2 in the train set the empirical distribution (histogram and pointwise Kernel Density Estimator) and SLGP-based estimation.}
	\label{fig:meteo_at_stations_without_data}
\end{figure}

However, further assessment of the estimation quality calls for principled approaches for the evaluation of probabilistic predictions of probability density functions, a topic that seems in its infancy. Yet, visual inspections suggest a fair extrapolation skill with some room for improvement, opening the way for further work on modelling and implementation choices. 

Finally, since the predictive capabilities of SLGP models are not limited to stations, we can perform statistical inference at any location in Switzerland. As a proof of concept, we now perform a joint quantile estimation, which is an immediate by-product of SLGP modelling. 

Directly visualising and aggregating these quantiles and uncertainties in 2D can be challenging. To provide clearer and more intuitive visualisations, we focus on an arbitrary slice crossing Switzerland. This slice allows us to explore the model’s predictions and uncertainties in a more interpretable way. The results are displayed in Figure~\ref{fig:SLGP_Meteo_slice_quantiles}

\begin{figure}[H]
	\centering
	\includegraphics[width=1\linewidth]{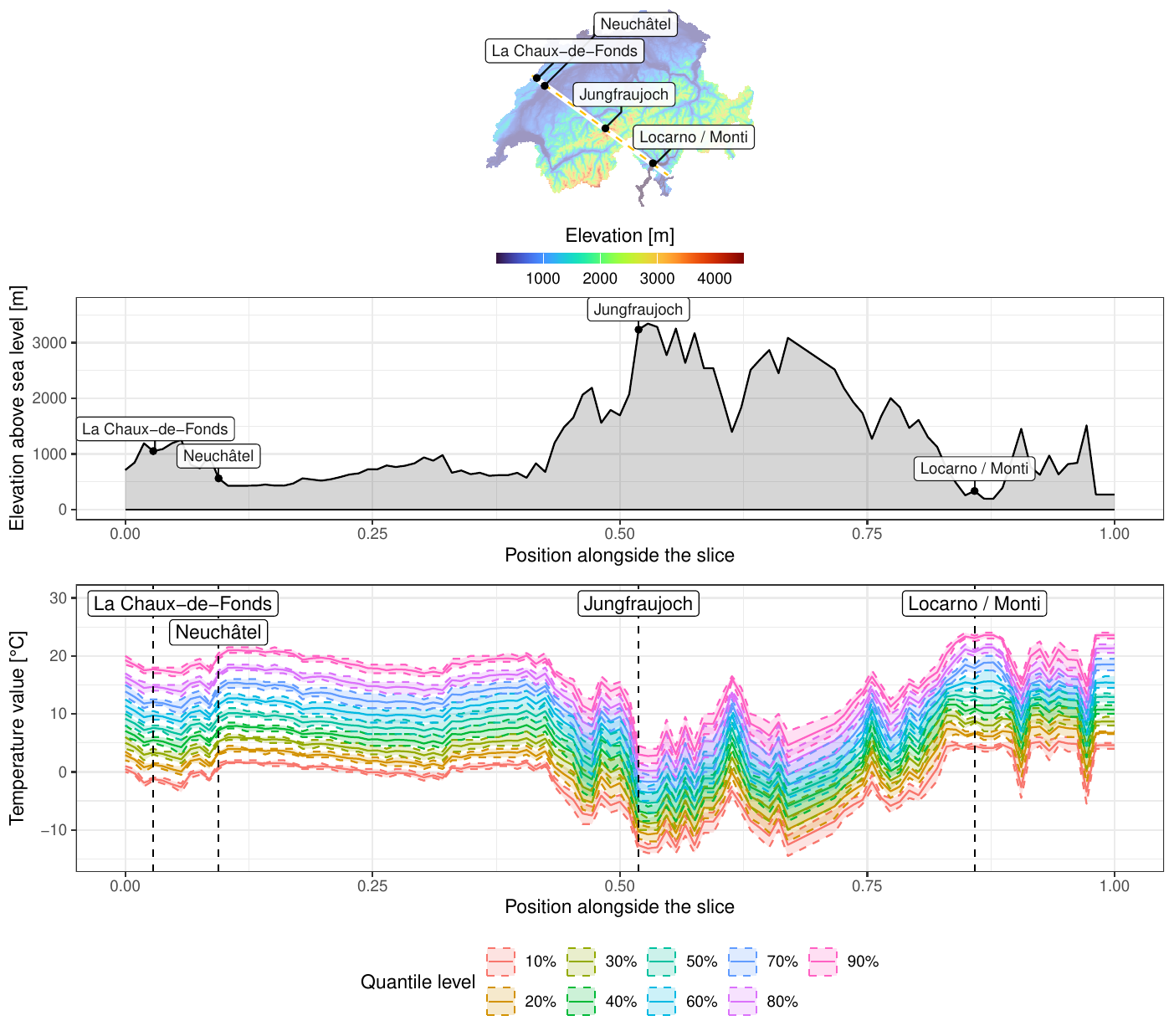}
	\caption{[Top] Map of Switzerland and slice considered. [Middle] Elevation alongside the slice and Stations location. [Bottom] Simultaneous quantile prediction (mean estimated value and 10\%-90\% quantile bands) across a slice of Switzerland.}
	\label{fig:SLGP_Meteo_slice_quantiles}
\end{figure}


The application is visually and intuitively convincing and did not present any major issues. In particular, SLGP models are able to capture trends on the temperature distributions, notably by identifying that temperature decreases when the altitude increases \footnote{According to the Swiss federal office for meteo, with every 100 metres, the temperature drops by an average of 0.65°C. \href{https://www.meteoswiss.admin.ch/weather/weather-and-climate-from-a-to-z/temperature/decreases-in-temperature-with-altitude.html}{https://www.meteoswiss.admin.ch/weather/weather-and-climate-from-a-to-z/temperature/decreases-in-temperature-with-altitude.html}}. This highlights that even when using a really simple model, with the altitude treated as an additional indexing variable rather than being incorporated as a covariate in the trend, SLGP models still yield promising insights in their generalisation abilities.

We believe that these could be further improved through the use of more appropriate kernels. Additionally, incorporating trends represents a meaningful direction for future research to enhance the realism of SLGP models. Nonetheless, our decision to adopt a general approach was motivated by the intention to provide a broadly applicable demonstration, rather than a solution overly tailored to this specific application.

\section{Conclusion and discussions}

In this paper, we investigated a class of models for non-parametric density field estimation. We revisited the Logistic Gaussian Process (LGP) model from the perspective of random measures, thoroughly investigating its relation to the underlying random processes and achieving some characterization in terms of increment covariances. We then built upon these investigations to further study Spatial logistic Gaussian process models, with a focus on their relationships with random measure fields. We demonstrated that when the underlying random field is continuous, SLGPs are characterised among random field of probability densities by the Gaussianity of associated log-increments.

Due to this particular structure, the SLGPs inherit their spatial regularity properties directly from the (Gaussian) field of increments of their log. This allowed us to leverage the literature of Gaussian Processes and Gaussian Measures to derive SLGP properties. We extended the notion of mean-square continuity  to  random  measure  fields and  established  sufficient  conditions  on  covariance  kernels underlying SLGPs for associated models to enjoy such properties.

We presented an implementation of the density field estimation that accommodates the need to estimate hyper-parameters of the model as well as the computational cost of this estimation. Our approach relies on Random Fourier Features, and we demonstrated its applicability of this approach in synthetic cases as well as on a meteorological application.

Several directions are foreseeable for future research. Extensive posterior consistency results as well as contraction rates are already available for logistic Gaussian processes \citep{tokdar_posterior_2007, van_der_vaart_rates_2008} and could be extended to the setting of SLGPs. We note that the posterior consistency is already studied when the index $\xX$ is considered as random (as presented in 
Section 3 of the supplementary material \citep{gautier_supplementary_2023}
, following Pati et al. 2013 \citep{pati_posterior_2013}), but the case of deterministic $\xX$ is still to be studied, as well as the posterior contraction rates. An analysis of the misspecified setting could complete such analysis \citep{kleijn_misspecification_2006}. Our study of the spatial regularity properties of the SLGP could also be complemented, as noted in Remark \ref{rmrk:uniform_approximation_GP}. 

In addition, further work towards evaluating the predictive performance of our model is needed. So far, we only relied on a squared integrated Hellinger distance for analytical settings where the reference field was known, and on qualitative and visual validation for the meteorological application. Comparing samples to predicted density fields calls for investigations in the field of scoring probabilistic forecasts of (fields of) densities. Therefore, evaluating the performances of the SLGP under several kernel settings, and comparing them to each other and to baseline methods is another main research direction.

As for implementation, we already highlighted in the Section \ref{sec:app}, that the current implementation leaves room to improvement. So far, choosing the basis functions in the finite rank implementation, the hyper-parameters prior, the trend of the GP, and even the range $\xI$ to consider was done mostly on the basis of expert knowledge or trial and errors. We would greatly benefit from development in the methodology, as this might reduce some computational costs while improving the quality of predictions. A more efficient implementation would also allow us to use the SLGP model at higher scales (i.e. with more data points and higher dimensions).

The SLGP could be highly instrumental for Bayesian inference, even more so as the latter progress would be achieved. We already started exploring its applicability to stochastic optimisation \citep{gautier_goal-oriented_2021}, and are optimistic about the SLGP's potential for speeding up Approximate Bayesian Computations \citep{gautier_probabilistic_2020}. Indeed, the flexibility of the non-parametric model allows for density estimation and it allows us to generate plausible densities. Having a generative model proves to be very beneficial, as it allows, for instance to derive experimental designs \citep{bect_supermartingale_2019}.

\newpage 
\appendix
\section{Appendix: Details of the proofs in section~\ref{sec:cty}}
\label{app:fullproofs}
For these proofs, we will 
need to bound the entropy number coming into play in Dudley's theorem. For a set $S\subset T$ and $\epsilon >0$, we denote by $N(\epsilon, I, d_Z)$ the entropy number, i.e. the minimal number of (open) $d_Z$-balls of radius $\epsilon$ required to cover $S$.

\begin{lemma}
\label{lemma:bound_entropy_supnorm}
For $d \geq 1$, and $\xI$ a convex, compact subset of $\xR^d$, we recall that if $\epsilon \geq \text{diam}(\xI)$, then $N(\epsilon, \xI, \Vert \cdot \Vert _\infty)=1$. Let $\text{Vol}$ be the volume, and $B_{1}^d$ the $d$-dimensional unit ball, we have:
\begin{equation}
		\left(\frac{1}{\epsilon}\right)^{d} \frac{\text{Vol}(\xI)}{\text{Vol}(B_{1}^d)} \leq N(\epsilon, \xI , \Vert \cdot \Vert _\infty)
	\end{equation}
	Additionally, if $\epsilon < \text{diam}(\xI^d )$:
	\begin{equation}
	\label{eq:right_side_covering}
		 N(\epsilon, \xI , \Vert \cdot \Vert _\infty) \leq \left(\frac{4}{\epsilon}\right)^{d} \frac{\text{Vol}(\xI)}{\text{Vol}(B_{1}^d)} 
	\end{equation}

\end{lemma}

\begin{proof}[Proof of Proposition~\ref{prop:Sup_norm_dif_holder}]
	\label{proof:Sup_norm_dif_holder}
	For fixed $(\xX, \xX') \in D^2$ and $Z \sim \mathcal{GP}(0, k)$, we note $\dxx$ the canonical distance associated to $\proc{Z}{\xX, \cdot}{}- \proc{Z}{\xX', \cdot}{}$, defined by:
	\begin{equation}
		\dxx^2(t, t') = \mathbb{E} \left[ \left(   [\proc{Z}{\xX, t}{}- \proc{Z}{\xX', t}{}] - [\proc{Z}{\xX, t'}{}- \proc{Z}{\xX', t'}{}] \right)^2 \right] \forall (t, t') \in \xI^2
	\end{equation}
	Using the Hölder condition on $k$ we note that we have simultaneously:
		\begin{align}
			 \dxx^2(t, t') \leq 3C \Vert \xX - \xX' \Vert^{\alpha_1}_\infty \forall (t, t') \in \xI^2\\
			 \dxx^2(t, t') \leq 4C \vert t - t' \vert^{\alpha_2} \forall (t, t') \in \xI^2
		\end{align}

	By Dudley's theorem, we can write :
	\begin{equation}
		\begin{array}{ll}
			\xM & \leq 24 \displaystyle\int_0^\infty \sqrt{\log(N(\epsilon, \xI, \dxx))} \,d\epsilon \\
			& \leq 24 \displaystyle\int_0^{\diam}\sqrt{\log(N(\epsilon, \xI, \dxx))} \,d\epsilon \\
		\end{array} 
	\end{equation}
	where $\diam$ stands for the diameter of $\xI$ with respect to the canonical distance associated to $\dxx$. As $\ \dxx^2(t, t') \leq 4C \vert t - t' \vert^{\alpha_2}$, we can combine the bounds stated in Lemma (\ref{lemma:bound_entropy_supnorm}) with the inequality
	\[N(\epsilon, \xI, \dxx) \leq N\left(\left(\dfrac{\epsilon}{4C}\right)^{2/\alpha_2},   \xI, \Vert \cdot \Vert _\infty  \right)\] 
 It follows that
	\begin{equation}
			M(\xX, \xX')  \leq 24 \sqrt{\dfrac{2 \dimI}{\alpha_2}} \displaystyle\int_0^{\diam} \sqrt{\log \left( \dfrac{K}{\epsilon} \right) } \,d\epsilon
	\end{equation}
	where $K := C 4^{1+\alpha_2/2}\left( \dfrac{\text{Vol}(\xI)}{\text{Vol}(B_{1}^{\dimI})}\right) ^{\alpha_2/(2\dimI)}$ and $B_{1}^{\dimI}$ stands for the $\dimI$-dimensional unit ball for $\Vert\cdot\vert_\infty$. To further compute the right-hand term, we introduce the error function, defined as $\text{erf}(x) = \dfrac{2}{\sqrt{\pi}} \int_0^x e^{-t^2} \,dt$.
	\begin{gather}
			\displaystyle\int_0^{\diam} \sqrt{\log \left( \dfrac{K}{\epsilon} \right) } \,d\epsilon 
			=  \left[ \epsilon \sqrt{\log \dfrac{K}{\epsilon } } -\dfrac{\sqrt{\pi}}{2} K \text{erf}\left( \sqrt{\log \dfrac{K}{\epsilon }}  \right) \right]^{\diam}_0\\
			= \diam  \sqrt{\log  \dfrac{K}{\diam }}  + K \displaystyle\int_{\sqrt{\log ( K/\diam ) } }^{\infty} e^{-t^2} \,dt
	\end{gather}
		Since for $y>0$, $\frac{2}{\sqrt{\pi}} \int_{y}^{\infty} e^{-t^2} \,dt \leq e^{-y^2}$, we also have:
				
	\begin{equation}
			M(\xX, \xX')  \leq 24 \sqrt{\dfrac{2 \dimI}{\alpha_2}} \left( \diam  \sqrt{\log \dfrac{K}{\diam } } +  \dfrac{\sqrt{\pi} \diam  }{2 K}   \right)
	\end{equation}
	Then, for any $0<\delta < \dfrac{\alpha_1}{2}$, $D$ being compact and considering that \[ y\left( \log(K/y ) \right)^{1/2} \underset{y \to 0}{=} o \left(y^{1- 2\delta / \alpha_1 } \right) \] we can conclude that there exists $K_\delta$ such that:
	\begin{equation}
		\begin{array}{rl}
			M(\xX, \xX')  \leq & K_\delta \dfrac{1}{(3 C )^{1/2 - \delta / \alpha_1}} \left( \diam^{1- 2\delta / \alpha_1 }   \right)\end{array} 
	\end{equation}
	
	Finally, by Equation~\ref{eq:Holder_d_kxx}, we have $\diam \leq \sqrt{3C} \xY $, and we can conclude that: 
	\begin{equation}
		\begin{array}{ll}
			M(\xX, \xX')  \leq K_\delta \Vert \xX - \xX' \Vert^{\alpha_1 /2 -\delta}
		\end{array} 
	\end{equation}
\end{proof}

\begin{proof}[Proof of Theorem \ref{th:Expected_quadratic_cty}]
\label{proof:Expected_quadratic_cty}
	Let us consider $(\xX, \xX') \in D^2$ and $\gamma > 0$.
	By Lemma~\ref{pty:bound_logtransf_sup}, there exists two constants $C_{KL}, C_{TV} > 0$ such that:
	
	\begin{equation}
		\label{eq:proof_cty_first}
		\begin{array}{c}
			\mathbb{E} \left[ 
			d_H(\proc{Y}{\xX, \cdot}, \proc{Y}{\xX', \cdot}{})^\gamma 
			\right] 
			\leq  \mathbb{E} 
			\left[ 
			\Vert \proc{Z}{\xX, \cdot}{} - \proc{Z}{\xX', \cdot}{} \Vert_\infty^\gamma e^{\Vert \proc{Z}{\xX, \cdot}{} - \proc{Z}{\xX', \cdot}{} \Vert_\infty \gamma/2}
			\right] \\			
			\mathbb{E} \left[ 
			KL(\proc{Y}{\xX, \cdot}{}, \proc{Y}{\xX', \cdot}{})^\gamma 
			\right] 
			\leq  C_{KL} \mathbb{E} 
			[f_1(	\Vert \proc{Z}{\xX, \cdot}{} - \proc{Z}{\xX', \cdot}{} \Vert_\infty) ]
			\\				
			\mathbb{E} \left[ 
			d_{TV}(\proc{Y}{\xX, \cdot}{}, \proc{Y}{\xX', \cdot}{})^\gamma 
			\right] 
			\leq  C_{TV} \mathbb{E} 
			[f_2(	\Vert \proc{Z}{\xX, \cdot}{} - \proc{Z}{\xX', \cdot}{} \Vert_\infty) ] 
		\end{array}
	\end{equation}
	where $f_1(x) = x^{2 \gamma} \left( 1+x \right)^{\gamma} e^{\gamma x}$ and $f_2(x) = x^{2 \gamma} \left( 1+x \right)^{2\gamma} e^{\gamma x}$.\\
	
	We consider the three functions, defined for $\gamma, M, y>0$:
	\begin{equation}
		\begin{array}{rl}
			f_{H, \gamma, M}(y) &= (My)^\gamma e^{\frac{M \gamma}{2} y}\\
			f_{KL, \gamma, M}(y) &= (My)^{2\gamma} (1+My)^{\gamma} e^{M \gamma y}\\
			f_{TV, \gamma, M}(y) &= (My)^{2\gamma} (1+My)^{2\gamma} e^{M \gamma y}\\
		\end{array}
	\end{equation}
	
	Then, if we consider $\xM = \mathbb{E} \left[ \Vert \proc{Z}{\xX, \cdot}{} - \proc{Z}{\xX', \cdot}{} \Vert_\infty \right] $, the previous inequalities can be rewritten as: 
	\begin{equation}
		\begin{array}{c}
			\mathbb{E} \left[ 
			d_H(\proc{Y}{\xX, \cdot}{}, \proc{Y}{\xX', \cdot}{})^\gamma 
			\right] 
			\leq \mathbb{E} \left[ 
			f_{H, \gamma, \xM}
			\left(
			\dfrac{\Vert \proc{Z}{\xX, \cdot}{} - \proc{Z}{\xX', \cdot}{} \Vert_\infty}{\xM}
			\right)
			\right] \\
			
			\mathbb{E} \left[ 
			KL(\proc{Y}{\xX, \cdot}{}, \proc{Y}{\xX', \cdot}{})^\gamma 
			\right] 
			\leq C_{KL} \cdot \mathbb{E} \left[ 
			f_{KL, \gamma, \xM}
			\left(
			\dfrac{\Vert \proc{Z}{\xX, \cdot}{} - \proc{Z}{\xX', \cdot}{} \Vert_\infty}{\xM}
			\right)
			\right] \\

			\mathbb{E} \left[ 
			d_{TV}(\proc{Y}{\xX, \cdot}{}, \proc{Y}{\xX', \cdot}{})^\gamma 
			\right] 
			\leq C_{TV} \cdot \mathbb{E} \left[ 
			f_{TV, \gamma, \xM}
			\left(
			\dfrac{\Vert \proc{Z}{\xX, \cdot}{} - \proc{Z}{\xX', \cdot}{} \Vert_\infty}{\xM}
			\right)
			\right]
		\end{array}
	\end{equation}
	
	By Fernique theorem (cf Proposition 2 of Section 1 in the supplementary material \citep{gautier_supplementary_2023} 
 ), there exists universal constant $\alpha$, $K > 0$, as well as $C_{H, \gamma, M}, C_{KL, \gamma, M}$ and $C_{TV, \gamma, M} > 0$  such that:
	
	\begin{equation}
		\begin{array}{c}
			\mathbb{E} \left[ 
			d_H(\proc{Y}{\xX, \cdot}{}, \proc{Y}{\xX', \cdot}{})^\gamma 
			\right] 
			\leq C_{H, \gamma, \xM} \cdot K \\
			
			\mathbb{E} \left[ 
			KL(\proc{Y}{\xX, \cdot}{}, \proc{Y}{\xX', \cdot}{})^\gamma 
			\right] 
			\leq C_{KL} \cdot C_{KL, \gamma, \xM} \cdot K \\

			\mathbb{E} \left[ 
			d_{TV}(\proc{Y}{\xX, \cdot}, \proc{Y}{\xX', \cdot})^\gamma 
			\right] 
			\leq C_{TV} \cdot C_{TV, \gamma, \xM} \cdot K
		\end{array}
	\end{equation}
	
	Detailed expressions of $C_{H, \gamma, M}, C_{KL, \gamma, M}$ and $C_{TV, \gamma, M}$ are given below this proof, and were derived with the tightness of our bounds in mind. We note that these coefficients seen as functions of $M$ are continuous, strictly positive for any $M > 0$ and that:
	\begin{equation}
		\begin{array}{c}
			C_{H, \gamma, M} \underset{M \to 0}{\sim} M^\gamma \left( \dfrac{\gamma}{2 \alpha} \right)^{\gamma / 2 } \exp\left\{ -\dfrac{\gamma}{2}\right\}\\
			
			C_{KL, \gamma, M} \underset{M \to 0}{\sim} M^{2 \gamma} \left( \dfrac{\gamma}{ \alpha} \right)^{\gamma } \exp\left\{ -\gamma \right\} \\
			
			C_{TV, \gamma, M} \underset{M \to 0}{\sim} M^{2 \gamma} \left( \dfrac{\gamma}{ \alpha} \right)^{\gamma } \exp\left\{ -\gamma \right\}  \\	
		\end{array}
	\end{equation} 
	This equivalence allows us to state that for a given $\gamma > 0$, there exists a rank $M_0 > 0$ and a constant $\kappa_{\gamma, 1} > 1$ such that for any $M < M_0$:
	\begin{equation}
		\begin{array}{ccc}
			C_{H, \gamma, M} \leq \kappa_{\gamma, 1}  M^\gamma, & C_{KL, \gamma, M} \leq \kappa_{\gamma, 1}  M^{2 \gamma}, & 	C_{TV, \gamma, M} \leq \kappa_{\gamma, 1}  M^{2 \gamma}
		\end{array}
	\end{equation} 
	
	We also observe that if $M$ is bounded, as $C_{H, \gamma, M}, C_{KL, \gamma, M}$ and $C_{TV, \gamma, M}$ seen as function of $M$ are continuous and strictly positive, there exists a constant $\kappa_{\gamma, 2}  > 0$ such that for values of $M \geq M_0$:
	\begin{equation}
		\begin{array}{ccc}
			C_{H, \gamma, M} \leq \kappa_{\gamma, 2}  M^\gamma, & C_{KL, \gamma, M} \leq \kappa_{\gamma, 2}  M^{2 \gamma}, & 	C_{TV, \gamma, M} \leq \kappa_{\gamma, 2}  M^{2 \gamma}
		\end{array}
	\end{equation} 
	
	Combining these two observations, and $\xM$ being bounded, we conclude that for any $\gamma > 0$ there exist $\kappa_{\gamma}$ such that:
	
	\begin{equation}
		\begin{array}{c}
			\mathbb{E} \left[ 
			d_{H}(\proc{Y}{\xX, \cdot}{}, \proc{Y}{\xX', \cdot}{})^\gamma 
			\right] \leq \kappa_{\gamma} \xM^{\gamma} \\
			\mathbb{E} \left[ 
			KL(\proc{Y}{\xX, \cdot}{}, \proc{Y}{\xX', \cdot}{})^\gamma 
			\right] \leq \kappa_{\gamma} \xM^{2 \gamma} \\
			\mathbb{E} \left[ 
			d_{TV}(\proc{Y}{\xX, \cdot}{}, \proc{Y}{\xX', \cdot}{})^\gamma 
			\right] \leq \kappa_{\gamma} \xM^{2 \gamma} \\
		\end{array}
	\end{equation}
	This argument relies on an equivalence at zero. Therefore, it ensures that the convergence rates are not degraded when bounding $C_{H, \gamma, M}, C_{KL, \gamma, M}$ and $C_{TV, \gamma, M}$, and that our bounds are still tight.
	
	Finally, using Proposition~\ref{prop:Sup_norm_dif_holder}, stating that for all $\delta > 0$, there exists $K_\delta$ such that:
	\begin{equation}
		\ M(\xX, \xX') \leq K_\delta \Vert \xX - \xX' \Vert^{\alpha_1 /2 -\delta}_\infty
	\end{equation}
	
	We conclude that for all $\gamma >0, \delta > 0$, there exists a constant $K_{\gamma, \delta}$ such that:
	
	\begin{equation}
		\label{eq:proof_cty_final}
		\begin{array}{c}
			\mathbb{E} \left[ 
			d_{H}(\proc{Y}{\xX, \cdot}{}, \proc{Y}{\xX', \cdot}{})^\gamma 
			\right] \leq K_{\gamma, \delta}  \Vert \xX - \xX' \Vert^{\gamma \alpha_1 /2 -\delta}_\infty  \\
			\mathbb{E} \left[ 
			KL(\proc{Y}{\xX, \cdot}{}, \proc{Y}{\xX', \cdot}{})^\gamma 
			\right] \leq K_{\gamma, \delta}  \Vert \xX - \xX' \Vert^{\gamma \alpha_1 -\delta}_\infty   \\
			\mathbb{E} \left[ 
			d_{TV}(\proc{Y}{\xX, \cdot}{}, \proc{Y}{\xX', \cdot}{})^\gamma 
			\right] \leq K_{\gamma, \delta}  \Vert \xX - \xX' \Vert^{\gamma \alpha_1 -\delta}_\infty \\
		\end{array}
	\end{equation}
\end{proof}

\begin{proof}[Finding the constants in Proof of Theorem \ref{th:Expected_quadratic_cty}]
For fixed $M$, $\gamma >0$, we consider the following three functions, for $x\geq 0$:
\begin{equation}
	\begin{array}{rl}
		f_{H, \gamma, M}(x) &= (Mx)^\gamma e^{\frac{M \gamma}{2} x}\\
		f_{KL, \gamma, M}(x) &= (Mx)^{2\gamma} (1+Mx)^{\gamma} e^{M \gamma x}\\
		f_{TV, \gamma, M}(x) &= (Mx)^{2\gamma} (1+Mx)^{2\gamma} e^{M \gamma x}\\
	\end{array}
\end{equation}

For $\alpha >0$,  we look for constants $C_{H, \gamma, M}, C_{KL, \gamma, M}, C_{TV, \gamma, M}$, satisfying:
\begin{equation}
	\begin{array}{rl}
		f_{H, \gamma, M}(x) & \leq C_{H, \gamma, M} e^{\alpha x^2}\\
		f_{KL, \gamma, M}(x) & \leq C_{KL, \gamma, M} e^{\alpha x^2}\\
		f_{TV, \gamma, M}(x) & \leq C_{TV, \gamma, M} e^{\alpha x^2}\\
	\end{array}
\end{equation}

such constants satisfy:
\begin{equation}
	\label{eq:ineq_C_expx2}
	\begin{array}{l}
		\sup\limits_{x \geq 0} \ g_{H, \gamma, M}(x) := f_{H, \gamma, M}(x) e^{-\alpha x^2} \leq C_{H, \gamma, M} \\
		\sup\limits_{x \geq 0} \ g_{KL, \gamma, M}(x):= f_{KL, \gamma, M}(x) e^{-\alpha x^2} \leq C_{KL, \gamma, M} \\
		\sup\limits_{x \geq 0} \ g_{TV, \gamma, M}(x) := f_{TV, \gamma, M}(x) e^{-\alpha x^2} \leq C_{TV, \gamma, M} \\
	\end{array}
\end{equation}

Studying the variations of $g_{H, \gamma, M}(x)$ simply involves finding the roots of a degree 2 polynomial and yields that this function attains its supremum at: 
\begin{equation}
	x_{H, \gamma, M} = \dfrac{M \gamma + 2 \sqrt{M^2 \gamma^2 + 8 \alpha \gamma}}{8 \alpha}
\end{equation}
Therefore a valid upper bound for :
	\begin{equation}
			C_{H, \gamma, M} = \left(M x_{H, \gamma, M} \right)^\gamma
			\exp\left\{ M \gamma \frac{x_{H, \gamma, M}}{2}- \alpha x_{H, \gamma, M}^2 \right\}
	\end{equation}
This constant is optimal in the sense that it is the smallest constant satisfying inequality \ref{eq:ineq_C_expx2}.

However, studying the variations of $g_{KL, \gamma, M}(x)$ and $g_{TV, \gamma, M}(x)$ is longer as it involves finding the roots of third degree polynomials. In order to simplify the constant, we use the simple property $1+x \leq e^x$ and introduce the bounds :
	\begin{equation}
		\begin{array}{l}
			g_{KL, \gamma, M}(x)\leq (Mx)^{2\gamma} e^{2 M \gamma x -\alpha x^2 } =: h_{KL, \gamma, M}(x) \\
			g_{TV, \gamma, M}(x)  \leq  (Mx)^{2\gamma} e^{3 M \gamma x -\alpha x^2 } =: h_{TV, \gamma, M}(x) \\
		\end{array}
	\end{equation}
These inequalities are tight at $Mx = 0$.

Studying the variations of $h_{KL, \gamma, M}(x)$ and $h_{TV, \gamma, M}(x)$ simply involves finding the roots of a degree 2 polynomial and yields that this function attains their supremum at: 

\begin{equation}
	\begin{array}{l}
		x_{KL, \gamma, M} = \dfrac{M \gamma + \sqrt{M^2 \gamma^2 + 4 \alpha \gamma}}{2 \alpha}\\
		x_{TV, \gamma, M} = \dfrac{3 M \gamma + \sqrt{9 M^2 \gamma^2 + 16 \alpha \gamma}}{4 \alpha}
	\end{array}
\end{equation}
Therefore, we can take the bounds:

	\begin{align}
			C_{KL, \gamma, M} = \left( M x_{KL, \gamma, M} \right)^{2\gamma} \exp\left\{ 2 M\gamma x_{KL, \gamma, M}  - \alpha x_{KL, \gamma, M}^2 \right\} \\
			C_{TV, \gamma, M} = \left( M x_{TV, \gamma, M} \right)^{2\gamma} 
			\exp\left\{ x_{TV, \gamma, M}  - \alpha x_{TV, \gamma, M} ^2 \right\} 
	\end{align} 
These bounds are tight around zero.
\end{proof}

\section*{Acknowledgements}
This contribution is supported by the Swiss National Science Foundation project number 178858. Calculations in section~\ref{sec:app:subsec:an} were performed on UBELIX (\url{http://www.id.unibe.ch/hpc}), the HPC cluster at the University of Bern.\\
The authors would like to warmly thank Wolfgang Polonik for his time and his review of an earlier version of this paper. They are also grateful to Ilya Molchanov and Yves Deville for their knowledgeable and insightful comments.

\bibliographystyle{spbasic}      
\bibliography{references.bib}   

\includepdf[pages=-]{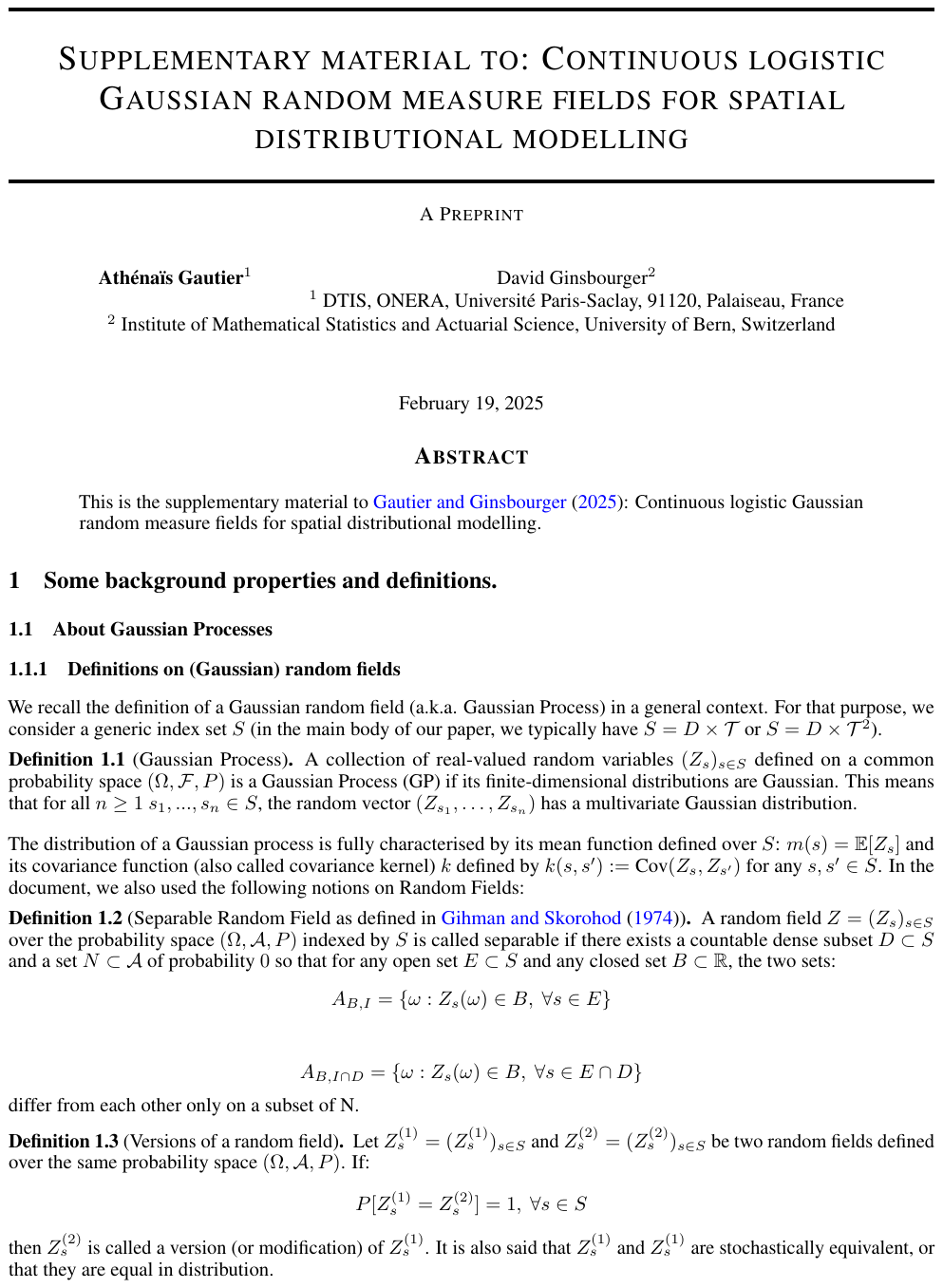}
\end{document}